\documentclass[a4paper,12pt,reqno]{amsart}

\usepackage{amsthm}
\usepackage{amsmath}
\usepackage{amssymb}
\usepackage{mathrsfs}
\usepackage{latexsym}
\usepackage{exscale}
\usepackage{geometry}
\usepackage{graphicx}
\usepackage[utf8]{inputenc}
\usepackage{verbatim}

\usepackage{enumerate}

\usepackage{color}
\definecolor{red}{rgb}{1,0.1,0.1}
\definecolor{blue}{rgb}{0.1,0.1,1}
\definecolor{vb}{RGB}{160,32,240}

\theoremstyle{plain}

\newtheorem*{teo*}{Theorem}
\newtheorem*{prop*}{Proposition}

\numberwithin{equation}{section}
\newtheorem{teo}{Theorem}[section]
\newtheorem{lema}[teo]{Lemma}

\theoremstyle{remark}
\newtheorem{obs}[teo]{Remark}

\theoremstyle{definition}

\newtheorem*{mydef*}{Definition}

\calclayout

\allowdisplaybreaks

\begin{document}
	\title[harmonic analysis in nilmanifolds]{ Some harmonic analysis on commutative nilmanifolds}
	
	\author[A.~L.~Gallo]{Andrea L. Gallo}
	\address{A.~L.~Gallo \\ FaMAF \\ Universidad Nacional de C\'ordoba \\
		CIEM (CONICET) \\ 5000 C\'ordoba, Argentina}
	\email{andregallo88@gmail.com}
	
	\author[L.~V.~Saal]{Linda V. Saal}
	\address{L.~~Saal\\ FaMAF \\ Universidad Nacional de C\'ordoba \\
		CIEM (CONICET) \\ 5000 C\'ordoba, Argentina}
	\email{saal@mate.uncor.edu}

	\thanks{ The authors are  partially supported by
		CONICET and SECYT-UNC}

	
	\dedicatory{\today}
	
	

	\begin{abstract}
		 In this work, we consider a family of Gelfand pairs $(K \ltimes N, N)$ (in short $(K,N)$) where $N$ is a two step nilpotent Lie group, and $K$ is the group of orthogonal automorphisms of $N$. This family has a nice analytic property: almost all these 2-step nilpotent Lie group have square integrable representations. 
		 In this cases, following Moore-Wolf's theory, we find an explicit expression for the inversion formula of $N$, and as a consequence, we decompose the regular action of $K \ltimes N$ on $L^{2}(N)$. This result completes the analysis carried out by Wolf in \cite{Wo1}, where the inversion formula is obtained in the case that $N$ has not square integrable representation. When $N$ is the Heisenberg group, we obtain the decomposition of $L^{2}(N)$ under the action of $K \ltimes N$ for all $K$ such that $(K,N)$ is a Gelfand pair. Finally, we also give a parametrization for the generic spherical functions associated to the pair $(K,N)$, and we give an explicit expression for these functions in some cases. 
	\end{abstract}
	
	\maketitle
	

\textbf{\bigskip }

\section{Introduction}
Let $G$ be a connected Lie group, and $K$ a compact subgroup of $G$. It is well known that the following are equivalent:
\begin{enumerate}
	\item The convolution algebra $L^{1}(K \backslash G/K)$ is commutative.
	\item The algebra $\mathcal{U}(\mathfrak{g})^{K}$ of $K$-invariant and left invariant differential operators on $G/K$ is commutative.
	\item The regular representation of $G$ on $G/K$ is multiplicity free.
	\item  For any irreducible unitary representation $\left( \rho ,\mathcal{H}%
	\right) $ of $K\ltimes N,$ the space $\mathcal{H}_{K}:=\left\{ v\in \mathcal{H%
	}:\rho \left( k\right) v=v\text{ for all }k\in K\right\} $ is at most one
	dimensional.
\end{enumerate}
When any of the above holds, we say that $(G,K)$ is a Gelfand pair.

Also, $G/K$ is called a nilmanifold if some nilpotent subgroup $N$ of $G$ acts transitively, and we say that $G/K$ is a commutative nilmanifold if $(G,K)$ is a Gelfand pair. In this work, $G/K$ is connected and simply connected, then $N$ acts simply transitively on $G/K$ and $G$ is the semidirect product $K \ltimes N$. We denote the Gelfand pair $(K \ltimes N,K)$ by $(K,N)$.

In the Vinberg's classification Theorem of commutative nilmanifolds, there is a big family that was defined by J. Lauret in \cite{L}. The Lauret's construction corresponds to the pairs where $K$ is the maximal orthogonal automorphism group, with exception of four cases. 
As we can see in \cite{Wo1}, in almost all the cases $N$ has a very surprising property: it has square integrable representation, with exception of three cases (two of which are in the Lauret's list). For this $N$, we develop the corresponding harmonic analysis, finding explicitly the inversion formula, and as a consequence we obtain the decomposition of the regular action of $G$ on $L^{2}(N)$. 
The remaining inversion formulas of the Lauret's list can  be found in \cite{Wo1}.\\

We now give a brief sumary of the results of the paper. In section 2, we introduce some preliminaries about the family described by J. Lauret, and some  results concerning nilpotent Lie groups.

In section 3, we present our main result and its proof. As a consequence we obtain a decomposition of the regular action on $L^{2}(N)$.

In section 4, we develop the harmonic analysis in the case  $(K,H_n)$, where $H_n$ is the $(2n+1)$-dimensional Heisenberg group.

In section 5, we describe the set of generic spherical functions associated to the Gelfand pair $(K,N)$.

\section{Preliminaries}

Let $\mathfrak{n}$ be a two step nilpotent Lie algebra with Lie bracket $%
\left[ \, \cdot, \cdot \, \right] $ and equipped with an inner product $\left\langle \, \cdot
, \cdot \, \right\rangle $. Then we write $\mathfrak{n=z \,\oplus \, } V$ where $\mathfrak{z}$
is its center and $V$ is the orthogonal complement of $\mathfrak{z.}$ Let $N$
be the connected simply connected Lie group with Lie algebra $\mathfrak{n,}$ and left invariant Riemannian metric determined by $%
\left\langle \, \cdot , \cdot \, \right\rangle.$ 

The group $N$ acts on $\mathfrak{n}$ by the
adjoint action $Ad$, and  $N$ acts on $\mathfrak{n}^{\ast }$, the dual space
of $\mathfrak{n}$, by the dual representation $Ad^{\ast }\left( n\right)
\lambda =\lambda \circ Ad\left( n^{-1}\right).$ Fixed a non trivial $%
\lambda \in \mathfrak{n}^{\ast },$ let $O\left( \lambda \right) :=\left\{
Ad^{\ast }\left( n\right) \lambda :n\in N\right\} $ be its coadjoint orbit.

We denote by $\widehat{N}$ the set of equivalence classes of irreducible
unitary representations of $N.$ From Kirillov's theory there is a
correspondence between $\widehat{N}$ and the set of coadjoint orbits.
Indeed, let 
\begin{equation} \label{corchetes}
B_{\lambda }\left( X,Y\right) :=\lambda \left( \left[ X,Y\right] \right), \, X,Y\in \mathfrak{n.}
\end{equation}
Let $\mathfrak{m}$ be a maximal isotropic subspace of $\mathfrak{n,}$ and
set $M=\exp \left( \mathfrak{m}\right) .$ Defining on $M$ the character $%
\chi _{\lambda }\left( \exp Y\right) =e^{i\lambda \left( Y\right) }$, the
irreducible representation corresponding to $O\left( \lambda \right) $ is
the induced representation $\rho _{\lambda }:=Ind_{M}^{N}\left( \chi
_{\lambda }\right) .$

Let $Z$ be the center of $N.$ Recall that an irreducible unitary
representation is called \textit{square integrable} if its matrix entries
are in $L^{2}\left( N/Z\right) .$ We denote by $\widehat{N}_{sq}$ the subset
of $\widehat{N}$ of square integrable classes. It follows from Moore-Wolf's theory that
\begin{enumerate}[(i)]
	\item If $\rho _{\lambda }\in $ $\widehat{N}$ has a matrix entry in $%
	L^{2}\left( N/Z\right) ,$ then $\rho _{\lambda }\in \widehat{N}_{sq}.$
	
	\item If $N$ has a square integrable representation then its Plancherel
	measure is concentrated on $\widehat{N}_{sq}.$
\end{enumerate}

We have that if $\rho _{\lambda }$ is a square integrable representation
then $B_{\lambda }$ is non degenerate on $V$ and the orbit is maximal,
that is $O\left( \lambda \right) =\lambda \mid_{\mathfrak{z}} \oplus  \, V^{\ast
}.$ Indeed, let $X_{\lambda }\in \mathfrak{z}$ be the representative of $\lambda {\mid }_{\mathfrak{z}}$, that is $\lambda \left( Y\right) =\left\langle Y,X_{\lambda
}\right\rangle $ for all $Y\in \mathfrak{z,}$ and denote by $\mathfrak{z}%
_{\lambda }$ the kernel of $\lambda {\mid}_{\mathfrak{z}}$. Let $\mathfrak{a}_{\lambda }$ be the subspace of $V$ where $B_{\lambda }$ is degenerate and let $\mathfrak{b}_{\lambda }$ be the complement of $\mathfrak{a}_{\lambda }$ in $V.$ Consider $\mathfrak{n}_{\lambda }=\mathfrak{
	a}_{\lambda } \oplus \, \mathfrak{b}_{\lambda } \oplus \, \mathbb{R}%
X_{\lambda }$ and $N_{\lambda }:=\exp \left( \mathfrak{n}_{\lambda }\right)
. $ We equip $\mathfrak{a}_{\lambda }$ with the trivial Lie bracket and $%
\mathfrak{h}_{\lambda }:=\mathfrak{b}_{\lambda } \oplus \, \mathbb{R}%
X_{\lambda }$ with Lie bracket 
\begin{equation*}
\left[ u,v\right] _{\mathfrak{h}_{\lambda }}=B_{\lambda }\left(
u,v\right) Y_{\lambda }, \ u,v\in \mathfrak{b}_{\lambda }, \ Y_{\lambda}:=\frac{X_{\lambda}}{|X_{\lambda|}}.
\end{equation*}

It is clear that $\mathfrak{h}_{\lambda }$ is a Heisenberg algebra and we
set $H_{\lambda }$ the corresponding Heisenberg group. We also set $%
A_{\lambda }:=\exp \left( \mathfrak{a}_{\lambda }\right).$ Since the
representation $\rho _{\lambda }$ is trivial on $\exp \left( \mathfrak{z}%
_{\lambda}\right),$ it factors through $N_{\lambda }.$ Identifying $%
N_{\lambda }$ with $A_{\lambda }\times H_{\lambda }$, we can write $\rho
_{\lambda }\left( a,n\right) =\chi \left( a\right) \rho _{\lambda }^{\prime
}\left( n\right) $ where $\chi $ is a unitary character of $A_{\lambda }$
and $\rho _{\lambda }^{\prime }$ is an irreducible representation of $%
H_{\lambda }.$ Thus $\rho _{\lambda }$ cannot be square integrable unless $%
\mathfrak{a}_{\lambda }=0.$

The reciprocal assertion is also true: if $B_{\lambda }$ is non
degenerate on $V$ (and thus $O_{\lambda }$ is maximal), then $\rho
_{\lambda }$ gives rise to an irreducible representation of $N_{\lambda }$
because $\lambda $ restricted to $\mathfrak{z}_{\lambda }$ is trivial. In
this case $N_{\lambda }$ is a Heisenberg group and every irreducible
representation of infinite dimension of $N_{\lambda }$ is square integrable,
so is $\rho _{\lambda }.$

This is a particular case of the following general result in the Moore-Wolf's theory. If $N$ is a connected simply connected nilpotent Lie group, the following are equivalent:
\begin{enumerate}[(i)]
	\item $\rho _{\lambda }$ is square integrable.
	\item The orbit $O_{\lambda \text{ }}$ is determined by $\lambda _{\mid } \mathfrak{z}$.
	\item $B_{\lambda }$ is non degenerate over $\mathfrak{n/z.}$\\
\end{enumerate}


The family to be considered in this work, was introduced by J. Lauret (see \cite{L}). Starting from a real representation $(\pi, V)$ of a compact Lie algebra $\mathfrak{g}= \mathfrak{c} \oplus \mathfrak{g}'$, where $\mathfrak{c}$ is the center of $\mathfrak{g}$ and $\mathfrak{g}' = [\mathfrak{g},\mathfrak{g}]$, let $\langle \cdot \,, \cdot \rangle_{\mathfrak{g}}$ and $\langle \cdot \,, \cdot \rangle_{V}$ be inner products on $\mathfrak{g}$ and $V$ respectively, such that $\langle \cdot \,, \cdot \rangle_{\mathfrak{g}}$ is $\mathit{ad}(\mathfrak{g})$-invariant and $\langle \cdot \,, \cdot \rangle_{V}$ is $\pi$-invariant. Let $\mathfrak{n}= \mathfrak{g} \oplus V$ and let $\langle \cdot \,, \cdot \rangle$ be the inner product in $\mathfrak{n}$ such that $\langle \cdot \,, \cdot \rangle_{\mathfrak{g} \times \mathfrak{g}} = \langle \cdot \,, \cdot \rangle_{\mathfrak{g}}$ and $\langle \cdot \,, \cdot \rangle_{V \times V} = \langle \cdot \,, \cdot \rangle_{V}$ with $\langle \mathfrak{g}\,, V \rangle = 0$. Such inner product $\langle \cdot \,, \cdot \rangle$ is called $\mathfrak{g}$-invariant.
The Lie algebra structure on $\mathfrak{n}$ is defined by assuming that $\mathfrak{g}$ is the center of $\mathfrak{n}$ and the Lie bracket on $V$ is given by 
\begin{equation}\label{corchete2}
\left\langle \left[ u,v\right] ,X\right\rangle =\left\langle \pi \left(
X\right) u,v\right\rangle \text{ for all }u,v\in V,X\in \mathfrak{g.}
\end{equation}

We denote by $N(\mathfrak{g},V)$ the connected simply connected Lie group with Lie algebra $\mathfrak{n}$. It is remarked that this construction does not depend of the $\mathfrak{g}$-invariant inner product $\langle \cdot\, , \cdot \rangle$ (up to Lie group isomorphism). Moreover, if $(\pi,V)$ and $(\pi',V')$ are two representations of $\mathfrak{g}$ and there exists an automorphism $\varphi$ of $\mathfrak{g}$ and an isomorphism $T: \, V \rightarrow V'$ such that $T\pi(x)T^{-1} = \pi'(\varphi(x))$  for all $x \in \mathfrak{g}$, then $N(\mathfrak{g},V)$ and $N(\mathfrak{g},V')$ are isomorphic Lie groups, see \cite{L}.

The group of orthogonal automorphisms of $N(\mathfrak{g},V)$ is $K=G' \times U$, where $G'$ is the connected simply connected Lie group with Lie algebra $\mathfrak{g}'=[\mathfrak{g}\,,\mathfrak{g}]$ and $U$ is the identity's connected component of the orthogonal group of intertwining operators of $(\pi,V)$. The component $U$ acts trivially on the center of $\mathfrak{n}$ and each $g \in G'$ acts on $\mathfrak{n}$ by $(Ad(g),\pi(g))$ where we also denote by $\pi$ the corresponding representation of $G'$ (see \cite{JL}, Theorem 3.12).

The group $N(\mathfrak{g},V)$ is said \textit{decomposable} if it is a direct product of Lie groups of the form 
$$N(\mathfrak{g},V)=N(\mathfrak{h}_1,V_1) \times N(\mathfrak{h}_2,V_2).$$
Otherwise we will say that $N(\mathfrak{g},V)$ is \textit{indecomposable}.
The list $A$ of Gelfand pairs of the form $(G' \times U,N(\mathfrak{g},V))$ where $N(\mathfrak{g},V)$ is indecomposable is the following: \\

\ \ \ \ \ \ \ \ \ \ \ \ \ \ \ \ \ \ \ \ \ \ \ \ \ \ \ \ \ \ \ \ \ \ \ \ \ \ \ \ \ \ \ \ \ \ \ \ \ \ \ \ \ \ \ \ \ \ \ \ \ \ \ \ \ \ \ \ \ \ 
\ \ \ \ \ \ \ \ \ \ \ \ \ \ \ \ \ \ \ \ \ \ \ \ \ \ \ \ \ \ \ \ \ \ \ $(A)$
\begin{enumerate}[(I)] 
	\item $(SU(2) \times Sp(n), N(\mathfrak{su}(2),(\mathbb{C}^{2})^{n})), \, n\geq1$, where $\mathfrak{su}(2)$ acts on $(\mathbb{C}^{2})^{n}$ as Im($\mathbb{H}$) acts component-wise on $\mathbb{H}^{n}$ by quaternion product on the left side, where $\mathbb{H}$ denotes the quaternions and Im($\mathbb{H}$) the imaginary quaternions. (Heisenberg type)\\
	\item $(SU(2) \times Sp(n),N(\mathfrak{su}(2), \mathbb{R}^{3} \oplus (\mathbb{C}^{2})^{n})), \, n \geq 0$, where $\mathfrak{su}(2)$ acts as $\mathfrak{so}(3)$ by rotations on $\mathbb{R}^{3}$, and $\mathfrak{su}(2)$  acts component-wise on ($\mathbb{C}^{2})^{n}$ in the standard way. \\
	\item $(Spin(4) \times Sp(k_1) \times Sp(k_2), N(\mathfrak{su}(2) \oplus \mathfrak{su}(2), (\mathbb{C}^{2})^{k_1} \oplus \mathbb{R}^{4} \oplus (\mathbb{C}^{2})^{k_2})), \, k_1+k_2 \geq 1$, where the real vector space $\mathbb{R}^{4}= (\mathbb{C}^2 \otimes \mathbb{C}^{2})_{\mathbb{R}}$ denotes the standard representation of $\mathfrak{so}(4)=\mathfrak{su}(2) \oplus \mathfrak{su}(2)$ and the first copy of $\mathfrak{su}(2)$ acts only on $(\mathbb{C}^{2})^{k_1}$ and the second one only on $(\mathbb{C}^{2})^{k_2}$.\\
	\item $(Sp(2) \times Sp(n), N(\mathfrak{sp}(2),(\mathbb{C}^{4})^{n}), \, n \geq 1$, where $\mathfrak{sp}(2)$ acts component-wise on $(\mathbb{H}^{2})^{n}$ in the standard way (identifying $\mathbb{H}^{2}$ with $\mathbb{C}^{4})$.\\

	\item $(SU(n) \times \mathbb{S}^{1}, N(\mathfrak{su}(n),\mathbb{C}^{n})), \, n \geq 3$, where $\mathbb{C}^{n}$ denotes the standard representation of $\mathfrak{su}(n)$ regarded as a real representation.\\
	\item $(SO(n),N(\mathfrak{so}(n),\mathbb{R}^{n})), \, n \geq 2$ (free two-step nilpotent Lie group), where $\mathbb{R}^{n}$ denotes the standard representation of $\mathfrak{so}(n)$.\\

	\item  $(U(n),N(\mathbb{R},\mathbb{C}^{n})), \, n \geq 1$ (Heisenberg group).\\
	\item $(SU(2) \times U(k) \times Sp(n), N(\mathfrak{u}(2),(\mathbb{C}^{2})^{k} \oplus (\mathbb{C}^{2})^{n})), \, k \geq 1, n \geq 0$, where the center of $\mathfrak{u}(2)$ acts non-trivially only on $(\mathbb{C}^{2})^{k}$, in fact, $(\mathbb{C}^{2})^{n}$ denotes the representation of $\mathfrak{su}(2)$ described in the item (I) and $\mathfrak{u}(2)$ acts component-wise on $(\mathbb{C}^{2})^{k}$ in the standard way.\\
	\item $(SU(n) \times \mathbb{S}^{1}, N(\mathfrak{u}(n),\mathbb{C}^{n})), \, n \geq 3$, where $\mathbb{C}^{n}$ denotes the standard representation of $\mathfrak{u}(n)$ regarded as a real representation.\\
	\item $(G' \times U,N(\mathfrak{g},V))$ where:\\
	
	\begin{itemize}
		\item [$-$] $\mathfrak{g} := \mathfrak{su}(m_1) \oplus \cdots \oplus \mathfrak{su}(m_{\beta}) \oplus\mathfrak{su}(2) \oplus \cdots \oplus \mathfrak{su}(2) \oplus \mathfrak{c}$, with $\alpha$ copies of $\mathfrak{su}(2), \, m_i \geq 3$ for all $1 \leq i \leq \beta$ and $\mathfrak{c}$ is an abelian component.
		
		\item [$-$] $V:=\mathbb{C}^{m_1} \oplus \cdots \oplus \mathbb{C}^{m_{\beta}} \oplus \mathbb{C}^{2k_1+2n_1} \oplus \cdots \oplus \mathbb{C}^{2k_{\alpha}+2n_{\alpha}}$, where $k_j \geq 1$ and $n_j \geq 0$ for all $1 \leq j \leq \alpha$.
		
		\item [$-$] $\mathfrak{g}$ acts on $V$ as follows: for each $1 \leq i \leq \beta+\alpha, \, \mathfrak{c}$ has a maximal subespace, denoted by $\mathfrak{c}_i$, and dim($\mathfrak{c}_i$)=1, acting non-trivially only on $\mathbb{C}^{m_i}$ (as the representation stated in item (V)) and for $\beta +1 \leq i \leq \beta + \alpha, \, \mathfrak{su}(2) \oplus\mathfrak{c}_i$ acts non-trivially only on  $\mathbb{C}^{2k_i+2n_i}$ (as the representation stated in the item (VIII)).
		
		\item [$-$] $U:=\mathbb{S}^{1} \times \cdots \mathbb{S}^{1} \times U(k_1) \times sp(n_1) \times \cdots\times U(k_{\alpha}) \times Sp(n_{\alpha})$, with $\beta$ copies of $\mathbb{S}^{1}$.
	\end{itemize}
\end{enumerate}

\medskip

Here we only consider the groups $N\left( 
\mathfrak{g},V\right) $ which have \textit{square integrable representation}; this condition holds for almost all $N\left( \mathfrak{g},V\right) $ in the family, with exception
of two cases: case II and case VI with $n$ odd  as is shown in \cite{Wo}, pages 339--341.  From Moore-Wolf's theory it follows
that the Plancherel measure $\nu $ for $N\left( \mathfrak{g},V\right)$
is concentrated on the equivalent classes of square integrable
representations. Moreover, they are parametrized by the elements of the dual
space $\mathfrak{g}^{\ast }$ of $\mathfrak{g,}$ and $\nu \left( \lambda
\right) =c\left\vert P\left( \lambda \right) \right\vert d\lambda $, where $P$ is a polynomial function, $d\lambda $ is
the Lebesgue measure on $\mathfrak{g}^{\ast }$ and $c$ is a specific
constant (Theorem 14.2.14 in \cite{Wo}).

We denote by $N$ the group $N(\mathfrak{g},V)$, and let $\left( \rho _{\lambda },\mathcal{H}%
_{\lambda }\right) $ be the irreducible representation of $N$ corresponding to $%
\lambda$, with $\lambda \in \mathfrak{g}^{\ast}$. 
For $k \in K$, let $\rho_{\lambda}^{k}(n):=\rho_{\lambda}(k\cdot n)$. So $\rho_{\lambda}^{k}$ is another irreducible representation of $N$ acting on $\mathcal{H}_{\lambda}$, and the stabilizer of $\rho_{\lambda}$ is $$K_{\rho_{\lambda}}:= \{k \in K \ : \ \rho_{\lambda}^{k} \ \text{is equivalent to  } \rho_{\lambda}\}.$$
Thus, for $k \in K_{\rho_{\lambda}}$ there exists a unitary operator  $\varpi \left( k\right) $ which intertwines $\rho _{\lambda }$ and $\rho^{k} _{\lambda }$. This gives rise to a non projective representation $\varpi _{\lambda }$ (see Theorem 2.3 in \cite{BJR2}) of $K_{\rho_{\lambda}}$ called the \textit{metaplectic }representation.

For fixed $\lambda \in \mathfrak{g}^{\ast}$, let $\mathcal{H}_{\lambda }=\oplus _{j\in \Lambda }W_{\lambda
	,j}$ be the decomposition of $\varpi _{\lambda }$ into irreducible $%
K_{\rho_{\lambda}}$-modules. For $j \in \Lambda$, let $d_j=dim(W_{\lambda,j})$ and let $\{v_l^{j} \}_{l=1}^{d_j}$ be an orthonormal basis of $W_{\lambda,j}$. We define

\begin{equation}\label{numerito}
\psi_{{\lambda},j}(n) = \sum_{l=1}^{d_j} \langle \rho_{\lambda}(n) v_l^{j}, v_l^{j} \rangle.
\end{equation}

There is an action of $K$ on $\mathfrak{g}^{\ast }$ defined by $\left( k\cdot\lambda \right) \left( X\right) =\lambda \left( k^{-1} \cdot X\right)$, then $K_{\rho_{\lambda}}=\{k \in K \, | \, k \cdot \lambda = \lambda \}$. Moreover, if $X_{\lambda }$ is the vector in $\mathfrak{g}$ such that $\lambda \left(X\right) =\left\langle X,X_{\lambda }\right\rangle $ for all $X\in \mathfrak{%
	g}$, clearly we have that $K_{\rho_{\lambda}}$ coincides with the stabilizer of $X_{\lambda}$, $K_{\lambda }:=\left\{ k\in K:\, k\cdot X_{\lambda }=X_{\lambda
}\right\} .$  Identifying $\lambda$ with the corresponding $X_{\lambda}$, we can assume that the Plancherel measure is defined on $\mathfrak{g}$ instead of $\mathfrak{g}^{\ast}$. Also, if we denote by $\mathfrak{g}'
_{r}$ the set of regular elements of $\mathfrak{g}'$, since the complement of $%
\mathfrak{g}'_{r}$ in $\mathfrak{g}'$ has Lebesgue measure zero, we can consider that the Plancherel measure $\nu$ is defined on $\mathfrak{g}'_r \oplus \mathfrak{c}$, where $\mathfrak{c}$ is the center of $\mathfrak{g}$. 
\medskip



If $Y_{\lambda}= \frac{X_{\lambda}}{|X_{\lambda}|}$, let $N_{\lambda }$ be the Heisenberg group with Lie algebra $\mathfrak{n}
_{\lambda } = \mathbb{R} Y_{\lambda } \oplus \, V$ and Lie bracket 
\begin{equation*}
\left[ u,v\right]_{\lambda} =B_{\lambda }\left( u,v\right) Y_{\lambda }, \, u,v\in V.
\end{equation*}
Notice that for $k\in K_{\lambda }$ and $u,v\in V,$ we have $B_{\lambda
}\left( k \cdot u,k \cdot v\right) =\left\langle X_{\lambda },\left[ k \cdot u,k \cdot v\right]
\right\rangle =\left\langle X_{\lambda },k \left[ u,v\right] \right\rangle
=\left\langle k^{-1} X_{\lambda },\left[ u,v\right] \right\rangle =B_{\lambda }\left( u,v\right) .$ Thus $K_{\lambda }$ is contained in the
symplectic group $Sp\left( B_{\lambda }\right) .$

Let $\mathfrak{z}_{\lambda}=Ker(\lambda \mid_{\mathfrak{g}})$. Furthermore, since $\lambda $ restricted to $\mathfrak{z}_{\lambda }$ is
trivial, $\rho _{\lambda }$ is an irreducible representation of $N_{\lambda
}, $ and the metaplectic action of $K_{\rho _{\lambda }}$ coincides with the
metaplectic action of $K_{\lambda }$. Moreover, if $\pi _{s}$ denotes the irreducible representation of $%
N_{\lambda }$ such that $\pi _{s}\left(
t,0\right) =e^{ist}$ realized on the Fock space of holomorphic (resp. antiholomorphic)
functions on $\mathbb{C}^{n}$ which are square integrable with respect to the measure $e^{-\left\vert \lambda \right\vert \frac{\left\vert z\right\vert ^{2}}{2}}$, we have that
\begin{equation*}
\rho _{\lambda }\left( z,0\right) =\rho _{\lambda }\left( \left\langle
z,Y_{\lambda }\right\rangle Y_{\lambda },0\right) =e^{i\left\langle
	z,Y_{\lambda }\right\rangle \lambda \left( Y_{\lambda }\right) } = e^{i|\lambda| \langle z,Y_{\lambda} \rangle},
\end{equation*}%
that is 
\begin{equation*}
\rho _{\lambda }\left( z,0\right) =\pi _{\left\vert \lambda \right\vert
}\left( \left\langle z,Y_{\lambda }\right\rangle ,0\right).
\end{equation*}
Therefore $$\rho _{|\lambda| }\left( z,v\right) =\pi _{|\lambda|}\left( \left\langle z,Y_{\lambda }\right\rangle ,v\right).$$\\

\section{The main result}

Given the non degenerate form $B_{\lambda }\left( u,v\right) =\lambda
\left( \left[ u,v\right] \right) ,$ the Pfaffian $Pf\left( B_{\lambda} \mid_{V \times V} \right) $
is defined as the square root of the determinant of  $B_{\lambda
}\mid_{V \times V}$. Let $P(\lambda):=Pf(B_{\lambda} \mid_{V \times V})$. This function $P$ only depends of $\lambda\mid_{\mathfrak{g}}$ and so there is a homogeneous polynomial function on $\mathfrak{g}^{\ast}$, which we also denote by $P$, such that $P(\lambda)=P(\lambda\mid_{\mathfrak{g}})$ (see \cite{Wo}, page 333).

According to Moore-Wolf's theory we have that if $f$ is a Schwartz function and $n\in N$, then
\begin{equation*}
f\left( n\right) =c\int_{\mathfrak{g}^{\ast }}tr\left( \rho _{\lambda
}\left( f\right) \rho _{\lambda }\left( n\right) \right) \left\vert P\left(
\lambda \right) \right\vert d\lambda,
\end{equation*}
where $c$ is a specific constant (\cite{Wo}, page 334).

We can decompose the metaplectic action of $K_{\lambda }$ on the Fock space as
\begin{equation*}
\mathcal{F}_{\lambda }=\bigoplus _{j\in \Lambda }W_{\lambda ,j},
\end{equation*}
to obtain the function $\psi_{\lambda,j}$ defined as in \eqref{numerito}. By straightforward computation, it is easy to see that 
\begin{equation} \label{estrella}
f\left( n\right) =c\sum_{j\in \Lambda }\int_{\mathfrak{g}^{\ast }}f\ast \psi_{\lambda,j}(n)
 \left\vert
P\left( \lambda \right) \right\vert d\lambda .
\end{equation}
Identifying $\lambda $ with $X_{\lambda },$ we can perform the integral \eqref{estrella} on $%
\mathfrak{g}$ instead of $\mathfrak{g}^{\ast }.$ Furthermore from now on, we will use the notation $P(\lambda)$ (resp. $\rho(\lambda), \, d\lambda$, $\psi_{\lambda,j}$, etc) or $P(X_{\lambda})$ (resp. $\rho(X_{\lambda}), \, dX_{\lambda}$, $\psi_{X_{\lambda},j}$, etc)  interchangeably.

\begin{lema}\label{Ad}
	 If $g\in G'$ and $X_{\lambda }\in \mathfrak{g}$  then $P\left( Ad\left( g\right) X_{\lambda }\right) =P\left(
	X_{\lambda }\right)$. 
\end{lema}

\begin{proof}
	Since $B_{\lambda }\left( u,v\right) =\lambda \left( \left[ u,v\right]
	\right) =\left\langle \left[ u,v\right] ,X_{\lambda }\right\rangle ,$ we
	have that 
	\begin{eqnarray*}
	B_{Ad\left( g\right) X_{\lambda }}\left( u,v\right) &=& \left\langle \left[ u,v\right] ,Ad\left( g\right) X_{\lambda} \right\rangle \\
	&=& \left\langle Ad\left( g^{-1}\right) \left[ u,v\right], X_{\lambda} \right\rangle \\
	&=& \left\langle \left[ \pi \left( g^{-1}\right) u,\pi \left(
	g^{-1}\right) v\right] ,X_{\lambda }\right\rangle \\
	&=& \lambda \left( \left[ \pi \left( g^{-1}\right) u,\pi \left( g^{-1}\right) v\right] \right),
	\end{eqnarray*}
 where in the third equality we have used that $\left( Ad\left(
	g\right) ,\pi \left( g\right) \right) $ is an automorphism of $N$. Therefore $P\left( Ad\left( g\right) X_{\lambda }\right) =P\left( X_{\lambda
	}\right)$, as desired.
\end{proof}

\medspace

 Recall that $\mathfrak{g}'_{r}$ denotes the set of regular elements of $\mathfrak{g}'$. Since the complement of $\mathfrak{g}'_{r}$ has Lebesgue measure zero (the complement is a set of
zeros of polynomials),  in \eqref{estrella} we can  integrate on $\mathfrak{g}'_{r}$.\\

Let $T$ be a maximal torus of $G'$ with Lie algebra $\mathfrak{h}$. Denote by  $\mathfrak{g}'_C$  and $\mathfrak{h}_C$ the complexified Lie algebras of $\mathfrak{g}'$ and $\mathfrak{h}$ respectively, and by $\Delta$ the root system corresponding to $(\mathfrak{g}'_C, \mathfrak{h}_C)$. Let $%
\mathfrak{h}_{\mathbb{R}}=i\mathfrak{h}$ and let $\mathfrak{C}$ be a fixed Weyl chamber of $%
\mathfrak{h}_{\mathbb{R}}$. 

Let $\Phi: G' \times \mathfrak{h} \rightarrow \mathfrak{g}'$ be defined by $\Phi(g,H) = Ad(g)H $. This map is surjective since any $X \in \mathfrak{g}'$ is contained in a Cartan subalgebra and two Cartan subalgebras are conjugated by $Ad(g)$, for some $g \in G'$.

Moreover, it is easy to see that  $\Phi :G'\times \mathfrak{h}_{r}\mathfrak{\rightarrow }\mathfrak{g}'_{r}$ is surjective, and  $\Phi :G'/T\times \mathfrak{h}_{r}\mathfrak{%
	\rightarrow }\mathfrak{g}'_{r}$ is well defined and surjective.

Consider $\Phi_{\mathfrak{C}} :G'/T\times \mathfrak{C} \rightarrow \mathfrak{g}'_{r},$ to be the restriction of the function $\Phi$ define as above, that is  $\Phi_{\mathfrak{C}}(gT,iH)=Ad(g)H$ (we will denote by $\Phi_{\mathfrak{C}}=\Phi$ for short). Since $\mathfrak{h}_{r}=\cup i\mathfrak{C}_{j}$, where the union is taken over all of the Weyl chambers, and the Weyl group $W\left( T\right)$ permute them, $\Phi$ is surjective. Let us see that is also injective. Assume that there exists $g\in G',$ and $iH,iH_{1}\in \mathfrak{C}$ such that $Ad\left( g\right) H=H_{1}.$ Thus $H_{1}\in \mathfrak{h} \cap Ad\left(
g\right) \mathfrak{h.}$ Since $H_{1}$ is a regular element we have that $Ad\left(
g\right) \mathfrak{h=h.}$ Then $Ad\left( g\right) $ permutes the Weyl
chambers and $g\in N(T)$, the normalizer of $T$ in $G$. Since $Ad\left( g\right) $ fixes $i\mathfrak{C}, \, gT$ is the identity of $W\left( T\right)$ (see \cite{Wa}, page 76, Theorem 3.10.9).

Thus, the map $\Phi : G'/T \times \mathfrak{C} \rightarrow \mathfrak{g}'_r$ is a diffeomorphism, and a computation shows that det($d\Phi_{(g,iH)}$)$=(-i)^{\# \Delta} \prod_{\alpha \in \Delta} \alpha(H)$, (see \cite{Kn}, pages 547--549).\\

Set $\theta(H):=|det(d\Phi_{(g,H)})|$. By the change of variables we obtain


\begin{equation*}
f\left( n\right) =c\sum_{j\in \Lambda }\int_{\mathfrak{c}} \int_{G'/T}\int_{\mathfrak{C}}f\ast \psi
_{(Ad\left( g\right) H + Z) \, , \, j}\left( n\right) \left\vert P\left(
Ad\left( g \right) H + Z \right) \right\vert \theta(H) \ dH \ d\dot{g} \ dZ_.
\end{equation*}

\medspace

where $d\dot{g}$ denotes the $G'$-invariant measure on $G'/T$, and by Lemma \ref{Ad}

\begin{equation*}
f\left( n\right) =c\sum_{j\in \Lambda }\int_{\mathfrak{c}} \int_{\mathfrak{C}}f\ast \left( \int_{G'/T} \psi_{Ad\left( g \right) (H + Z) \, , \,j}\left( n\right) d\dot{g} \right) \left\vert
P\left( H + Z \right) \right\vert \theta(H) \ dHdZ.
\end{equation*}

Recall that for $k \in K$, $\rho^{k}_{\lambda}$ is the irreducible representation of $N$ corresponding to $k \cdot \lambda$, and thus if $\lambda \mid_{\mathfrak{g}}$ is represented by the vector $H+Z$ then $k \cdot \lambda$ corresponds to $k \cdot (H+Z)$. Since $\left( Ad\left( g\right) ,\pi \left( g\right) \right) $ is an
automorphism of $N,$ we have that $\rho_{Ad(g)(H+Z)} (n) = \rho^{Ad(g)}_{H+Z}(n)  = \rho_{H+Z}(Ad(g) \cdot n)$, so \ $\psi_{Ad(g)(H+Z) \, , \,j} \, (n) = \psi_{H+Z \, , \, j} \, (Ad(g) \cdot n)$.\\

Let $C_c^{\sharp}(N)$ be the algebra of $K$-invariant continuous functions on $N$ with compact support. We say that a $K$-invariant continuous function $\phi$ on $N$ is a \textit{spherical function} if the linear functional $\chi(f):=\int f(x) \phi(x^{-1}) \ dx$ is a non trivial character of $C_c^{\sharp}(N)$. It is well known that the set of bounded spherical functions can be identified with the homomorphisms of the space of the $K$-invariant integral functions on $N$ via the map $$\phi \longrightarrow \chi(f)=\int f(n) \phi(n^{-1}) \ dn.$$ We also have the following result

\begin{lema} \label{centro} 
	\begin{enumerate}[(i)]
 \item If $X_{\lambda } = H + Z$, with $H \in \mathfrak{g}'_{r} , \, H \neq 0$ and $Z \in \mathfrak{c}$, then $K/K_{\lambda
}=G'/T$. Moreover, $$\phi _{\lambda ,j}\left( n\right) :=\int_{G'/T}\psi _{Ad\left(
g\right) (H+Z) \, , \, j}\left( n\right) d\dot{g}$$ is a spherical function of $\left( K,N\right).$
\item If $X_{\lambda} \in \mathfrak{c}$, then $K_{\lambda}=K$. In particular, if $\lambda \in \mathfrak{c}^{\ast}$, $\phi_{\lambda,j}=\psi_{\lambda,j}$.
\end{enumerate}
\end{lema}

\begin{proof}
\begin{enumerate}[(i)]
\item Let $$C_{G'}\left( X_{\lambda }\right) = \left\{ g\in G':Ad\left( g\right) X_{\lambda }=X_{\lambda }\right\}$$ be the centralizer of $X_{\lambda }$
in $G'$.
Since $U$ acts on $%
\mathfrak{g}$ by the identity, $K_{\lambda }=C_{G'}\left( X_{\lambda
}\right) \times U$ and since $H$ is a regular element,
$C_{G'}\left( X_{\lambda }\right) = C_{G'}\left( H \right)  $ is a maximal torus of $G'.$

 The description of the bounded spherical functions of a Gelfand pair $(K,N)$ is given in Theorem 8.7 in \cite{BJR1}. Indeed, let $(\rho, \mathcal{H}_{\lambda}) \in \widehat{N}$, let $\mathcal{H}_{\lambda} = \oplus_{j \in \Lambda} W_{j,\lambda}$ be the decomposition of the metaplectic representation of $K_{\rho_{\lambda}}$ into irreducible components and let $\{v_1, \cdots, v_d\}$ be an orthonormal basis of $W_{\lambda,j}$. Then the proof of Theorem 8.7 shows that the spherical functions are given by
\begin{equation} \label{esfericas}
\phi_{\lambda,j}(n) = \int_{K/K_{\rho_{\lambda}}} \sum_{l=1}^{d_j} \langle \rho_{\lambda}(\dot{k} \cdot n) v_l, v_l \rangle \, d \dot{k},
\end{equation}
where $\dot{k}$ denotes the $K$-invariant measure on $K/K_{\rho_{\lambda}}$.

In our case $K_{\rho_{\lambda}}=K_{\lambda}$, $K/K_{\lambda}=G'/T$ \ and $$\int_{G'/T} \psi_{Ad(g)(H+Z), \, j} (n) \, d\dot{g} = \int_{G'/T} \psi_{H+Z \, , \, j}(Ad(g) \cdot n) \, d\dot{g} 
= \int_{K/K_{\lambda}} \psi_{H+Z \, , \, j} (k \cdot n) \, d \dot{k},$$
which implies the assertion $(i)$.\\

\item As above, $K_{\lambda} = C_{G'}(X_{\lambda}) \times U$, and since $X_{\lambda} \in \mathfrak{c}$, $C_{G'}(X_{\lambda})=G'$, then $K_{\lambda}=K$.
\end{enumerate}
\end{proof}

We denote by $V_C$ the complexification of $V$ and by $(\pi_C, V_C)$ the extension of $\pi$ to $\mathfrak{g}_C$. Let $V_C = \oplus_r W_r$ the decomposition into irreducible representations and $W_r = \oplus_{j}W^{\nu_r^{j}}$ the weight space decomposition, that is
$$W^{\nu_r^{j}}:=\{v \in W_r \ | \ \pi_C(H)(v) = \nu_r^{j}(H)v \ \text{for all } H \in \mathfrak{h}_C \}.$$
Then, for $H \in \mathfrak{C}$, we have $$P(H)=\prod_{r,j} |\nu_r^{j}|^{m_r^{j}/2},$$ where $m_r^{j}$ is the dimension of $W_r^{j}$.
Let $\zeta_r$ be the central character of $\pi_C \mid_{W_r}$. Also, for $Z \in \mathfrak{c}$ and $H \in \mathfrak{C}$, we have 
\begin{equation}\label{pfaffiano}
P(H+Z)=\prod_{r,j} \nu_r^{j}(H) + \zeta_r(Z)|^{m_r^{j}/2}
\end{equation}

Hence, we have proved our main result:
\begin{teo} \label{teo} Let $f$ be a Schwartz function on $N$. Then
\begin{equation*}
f\left( n\right) =c\sum_{j\in \Lambda }\int_{\mathfrak{c}} \int_{C}f\ast \phi _{\lambda
,j}\left( n\right) \left\vert P\left( H+ Z \right) \right\vert \, \theta(H) \, dH \, dZ,
\end{equation*}
where $\phi_{\lambda,j}$ is the spherical function defined as in \eqref{esfericas} and the function $P$ is as \eqref{pfaffiano}. The support of the Plancherel measure $\nu$ is $\Lambda \times \mathfrak{C} \times \mathfrak{c}$, and $\nu$ is given by the product of the counting measure and $d\mu(\lambda)=|P(H+Z)| \, \theta(H) \, dH \, dZ$. 
\end{teo}

We write $\lambda = \lambda' + \lambda_0$, with $\lambda' \in [\mathfrak{g}, \mathfrak{g}]^{\ast}$, and $\lambda_0 \in \mathfrak{c}^{\ast}$. As a consequence of the previous result we obtain the decomposition of the regular action on $L^{2}(N)$.

\begin{teo} \label{teonil} 
	
	Let $\mathfrak{g}$ be any compact Lie algebra that appears in the list $A$. Then the regular action of $K \ltimes N$ on $L^{2}(N)$ decomposes as a direct integral of irreducible components by
	
	\begin{equation*}\label{decomposition}
	L^{2}\left( N\right) =\sum_{j\in \Lambda } \int_{\mathfrak{c}} \int _{\mathfrak{C}} \mathcal{H}_{\lambda \, , \, j%
		\text{ }}  \ d\mu \left( \lambda \right),
	\end{equation*}
	where $\mu$ is the measure $\mu \left( \lambda \right) = \left\vert P\left( \lambda \right) \right\vert \theta(\lambda') \, d\lambda $ and $d\lambda$ is the Lebesgue measure in $\mathfrak{c} \times \mathfrak{C}$. Moreover, the projection over $\mathcal{H}_{\lambda,j}$ is $Q_{\lambda,j}(f)=f \ast \phi_{\lambda ,j}$, where $\phi_{\lambda,j}$ is the spherical function given by the following:
	\begin{enumerate}[(i)]
		\item \ If $\lambda' \neq 0$, 
		\begin{equation}\label{esffunc}
		\phi _{\lambda ,j}\left( n\right) =\int_{G'/T } \psi_{\lambda,j}(\dot{g}\cdot n) \, d\dot{g},
		\end{equation}
		where $g \cdot n$ denotes the action of $G'$ by automorphism on $N$, $d \dot{g}$ is the $G'$-invariant measure on $G'/T$ and $\psi_{\lambda.j}$ is as in \eqref{numerito}. 
		
		\medskip
		
		\item \ In the case VII, $\mathfrak{g}= \mathbb{R}$, and $\phi_{\lambda ,j}= \psi_{\lambda,j}$ with $\psi_{\lambda,j}$ as in \eqref{numerito}.\\
		
		If $\lambda' = 0$ and $\mathfrak{g}$ belongs to the case VIII with $k \geq 1$ and $n=0$, case IX and case X with $k_j \geq 1$ and $n_j=0$ for all $1 \leq j \leq \alpha$, then $\phi_{\lambda ,j}= \psi_{\lambda,j}$ with $\psi_{\lambda,j}$ as in \eqref{numerito}.\\
		
		In the other cases, the Plancherel measure vanishes on $\mathfrak{c}$.
		
	\end{enumerate}
\end{teo}

\begin{proof} 	
\begin{enumerate}[(i)]
		\item It follows from Theorem \eqref{teo}.\\

		\item 	In the cases I, III, IV ,V and VI with $n$ even $\mathfrak{g}$ is semisimple and it has trivial center. 
		
		In the case IX $ \mathfrak{g}= \mathfrak{u}(n)= \mathfrak{su}(n) \oplus i\mathbb{R}, V=\mathbb{C}^{n}, \, n \geq 3$, where $\mathbb{C}^{n}$ denotes the standard representation of $\mathfrak{u}(n)$. Since $Ker(\pi(X_{\lambda}))$ is trivial for all $X_{\lambda} \in \mathfrak{u}(n)$, it follows that $B_{\lambda}$ is non degenerate. Then, the Plancherel measure is concentrated in  $\mathfrak{g}= i\mathbb{R} \oplus [\mathfrak{g},\mathfrak{g}]$. The expression of the spherical functions follows from Lemma \ref{centro} (ii).
		
		In the case VIII with $k \geq 1$, $ n=0$, and case X with $k_j \geq 1$, $n_j =0$ for all $1 \leq j \leq \alpha$ the analysis is similar to the case IX since $\pi$ has trivial kernel.

		In the case VIII \ with $k \geq 1, n > 0$, $\mathfrak{g}=\mathfrak{u}(2)= \mathfrak{su}(2) \oplus i \mathbb{R}$, $V=(\mathbb{C}^{2})^{k} \oplus (\mathbb{C}^{2})^{n}$. The center of $\mathfrak{u}(2)$ acts non-trivially only on $(\mathbb{C}^{2})^{k}$, in fact, $\mathfrak{su}(2)$ acts on $(\mathbb{C}^{2})^{n}$ as Im($\mathbb{H}$) acts component-wise on $\mathbb{H}^{n}$ by quaternion product on the left side. Thus, if $t \in \mathbb{R}$, $\pi(it)(0,v)=(0,0)$ for all $v \in (\mathbb{C }^{2})^{n}$, that is, $(0,v) \in Ker(\pi(it))$. For \eqref{corchete2} and \eqref{corchetes} it follows that $B_{it}$ is degenerate for all $it \in i \mathbb{R}$. Then, by Theorem 14.2.10 in \cite{Wo}, the Plancherel measure is concentrated in $\mathfrak{g}'=[\mathfrak{g},\mathfrak{g}]$.
		
		In the case X with $k_j \geq 1$ for all $1 \leq j \leq \alpha$ and $n_{j_0} > 0$ for some  $1 \leq j_0 \leq \alpha$ the analysis is similar to the case VIII with $n>0$.
\end{enumerate}
        The case VII corresponds to the Heisenberg group, and it is proved in section 4, Theorem \ref{teoHeis}.
\end{proof}

 \section{The Heisenberg case}

 We take $\mathfrak{g}=\mathbb{R}$ and $V=\mathbb{C}^{n}$ with the standard
 Hermitian form $\left( u,v\right) =Re(\sum_{i=1}^{n}u_i\overline{v}_i)$, where $u_i, v_i$ are the coordinates of $u,v \in \mathbb{C}^{n}$ respectively and let $\pi$ defined by $\pi \left(
 t\right) v=itv,$ for $t\in \mathbb{R}$, in this case we have that
 \begin{equation*}
 \left\langle t,\left[ u,v\right] \right\rangle =\left( \pi \left( t\right)
 u,v\right) =t\left( iu,v\right) =-t \, \text{Im}(u.\overline{v}).
 \end{equation*}%
 Thus, the bracket is given by the standard simplectic form and the
 corresponding group $N\left( \mathfrak{g,}V\right) $ is the $(2n+1)$-dimensional Heisenberg group.
 \medskip
 
 It is known that the unitary irreducible representations of $H_n$ are of two types: those of infinite dimension acting non trivially on the center and the characters $\chi_{w}(t,v)=e^{i \text{Re}(v.\bar{w})}$. The unitary irreducible representations of infinite dimension $( \pi
 _{\lambda },\mathcal{F}_{\lambda }) $ of $H_{n}$ are parametrized by $0\neq \lambda \in \mathbb{R}$. More explicitly, for $\lambda >0$ (resp. $\lambda <0),$ they are realized on the Fock space of holomorphic (resp. antiholomorphic)
 functions on $\mathbb{C}^{n}$ which are square integrable with respect to the measure $e^{-\left\vert \lambda \right\vert \frac{\left\vert z\right\vert ^{2}}{2}}$ and they are determined by the central action. We denote by $\mathcal{P}(V)$ the polynomial algebra which is dense in $\mathcal{F}%
 _{\lambda }$.
 
 Let $K\subseteq U\left( n\right) $ such that $\left( K,H_{n}\right) $ is a
 Gelfand pair. For $k\in K$ and $\left( t,v\right) \in H_{n}$, we can define $\pi _{\lambda
 }^{k}\left( t,v\right) :=\pi _{\lambda }\left( \left( t,kv \right) \right) .$
 Since $\pi _{\lambda }\left( t,0\right) =e^{i\lambda t}$, we have that $K=\left\{ k\in K:\pi
 _{\lambda }^{k}\sim \pi _{\lambda }\right\}$. For $p\in \mathcal{P}
 \left( V\right)$ and $k\in K$ we define
 \begin{equation}\label{metaplectica}
 \varpi \left( k\right) p\left( v\right) =p\left( k^{-1}v\right).
 \end{equation}
 Then $\varpi $ extends to a unitary representation of $K$, called the
 metaplectic representation, which intertwines $\pi _{\lambda }^{k}$ and $%
 \pi _{\lambda }.$ According to Mackey's theory, the irreducible unitary representations  of $K\ltimes H_{n}$ are induced by those of $H_n$. 
 
 For $\sigma \in \widehat{K}$, the irreducible representations of $K \ltimes H_n$ induced by $\mathcal{F}_{\lambda}$ are
 defined by $$\rho _{\lambda ,\sigma
 }\left( k,t,v \right) =\sigma \left( k\right) \otimes \varpi \left( k\right)
 \pi _{\lambda }\left( t,v \right), \ \ k \in K, \ (t,v) \in H_n. $$
 Thus $\rho _{\lambda
 	,\sigma }$ has a vector fixed by $K$ if and only if $\sigma $ is the dual
 representation of some irreducible component of $\varpi $. 
 
 Since the other elements of $\widehat{K\ltimes H_{n}}$ are induced by the
 characters of $\mathbb{R}^{2n},$ and $(K\ltimes \mathbb{R}^{2n},K)$ is always a
 Gelfand pair, we have that $\left( K,H_{n}\right) $ is a Gelfand pair if and
 only if $\varpi $ is multiplicity free. 
 
 Let 
 \begin{equation*}
 \varpi \downarrow \mathcal{F}_{\lambda }=\bigoplus _{j\in \Lambda }\varpi _{j}
 \end{equation*}%
 the decomposition of $\varpi $ into irreducible components. We denote by $%
 W_{j}$ the representation space of $\varpi _{j}$ and by $\varpi _{j}^{\prime
 }$ its dual representation.
 
 We select an orthonormal basis $\left\{ h_{1},...,h_{d_{j}}\right\} $ of $%
 W_{j},$ and let $\left\{ h_{1}^{\ast},...,h_{d_{j}}^{\ast}\right\} $  its
 dual basis. It follows immediately that $s_{j}=\sum_{l=1}^{d_{j}}h_{l}
 \otimes h_{l}^{\ast}$ is a vector of $\rho _{\lambda ,\varpi _{j}^{\prime }}$
 fixed by $K.$ In order to simplify the notation, we set $\rho _{\lambda
 	,j}:=\rho _{\lambda ,\varpi _{j}^{\prime }}.$ The spherical function
 corresponding to $s_{j}$ is $\phi _{\lambda ,j}\left( t,v\right)
 =\left\langle \rho _{\lambda ,j}\left( k,t,v\right) s_{j},s_{j}\right\rangle 
 $ and an easy computation gives that
 \begin{equation}\label{traza}
 \phi _{\lambda ,j}\left( t,v\right) =\sum_{i=1}^{d_j} \langle \pi _{\lambda }\left( t,v\right) h_i, h_i \rangle.
 \end{equation}
 
 For $h, h' \in \mathcal{F}_{\lambda}$, let $e_{\lambda }\left( h,h'\right) \left( t,v\right) :=\left\langle \pi
 _{\lambda }\left( t,v\right) h,h' \right\rangle$ the entry matrix of $\pi _{\lambda }$ associated to $h, h'$. It is well known that the functions $v \longmapsto e_{\lambda}(h,h')(0,v) \in L^{2}(\mathbb{C}^{n})$ and for $\lambda \neq \lambda_1$  
 \begin{equation*}
 \int_{\mathbb{C}^{n}}e_{\lambda }\left( h,h'\right) \left( 0,t\right) 
 \overline{e_{\lambda_1}\left( h_1,h_1^{\prime }\right) \left(
 	0,t\right) } dv=0
 \end{equation*}%
 for all $h,h' \in \mathcal{F}_{\lambda}, \, h_1,h_1' \in \mathcal{F}_{\lambda_1}$, and
 \begin{equation*}
 \int_{\mathbb{C}^{n}}e_{\lambda }\left( h_1,h_2\right) \left( t,v\right) 
 \overline{e_{\lambda}\left( h_3,h_4 \right) \left(
 	0,t\right) } dv= \langle h_1,h_3 \rangle \langle h_2,h_4 \rangle
 \end{equation*}%
 for all $h_1,h_2,h_3,h_4 \in \mathcal{F}_{\lambda}$ (see Proposition 1.42 in \cite{Fo} and Theorem 14.2.3 in \cite{Wo}).\\

 For $j\in \Lambda ,$ we select an orthonormal basis $\mathcal{B}_{j}$ of $%
 W_{j}$, hence $\mathcal{B=\cup }_{j}\mathcal{B}_{j}$ is an orthonormal basis
 of $\mathcal{F}_{\lambda }.$
 
 Recall that the convolution is defined for integrable functions on $H_n$ by
 $$\left(f \ast g \right) (x) = \int_{H_n} f(y) g(y^{-1}x) \ dy,$$ where $x \in H_n$.

 \begin{teo} \label{teoHeis}
 	Let $H_{\lambda ,j}$ be the Hilbert
 	space generated by $\left\{ e_{\lambda}\left( h_{\alpha },h_{\beta }\right) \right\}$ where $h_{\alpha }\in \mathcal{B}_{j}$ and $h_{\beta }\in \mathcal{B}$, with inner product $
 	\left\langle \varphi ,\psi \right\rangle _{\lambda ,j}:=\int_{\mathbb{C}^{n}}\varphi \left( 0,v\right) \overline{\psi \left( 0,v\right) }dv.$ Then, the regular action of $K\ltimes N$ decomposes as a direct
 	integral of irreducible components by
 	\begin{equation*}
 	L^{2}\left( H_{n}\right) =\sum_{j\in \Lambda }\int_{-\infty }^{\infty
 	}H_{\lambda ,j}\left\vert \lambda \right\vert ^{n}d\lambda \mathit{.}
 	\end{equation*}
 	Moreover, the Hilbert space $H_{\lambda,j}$ is primary and equivalent to $\left( \dim W_{j}\right) F_{\lambda }$ as  $H_{n}$-module.
 \end{teo}
 
 \begin{proof}
 	The inversion formula for a Schwartz function $f$ on $H_{n}$ is given by
 	\begin{eqnarray*}
 		&& \\
 		f\left( t,v\right) &=&\int_{-\infty }^{\infty }tr\left( \pi _{\lambda
 		}\left( t, v\right) \pi _{\lambda }\left( f\right) \right) \left\vert
 		\lambda \right\vert ^{n}d\lambda.
 	\end{eqnarray*}
 	Moreover, 
 	\begin{equation*} 
 	\left\Vert f\right\Vert ^{2}=\int_{-\infty }^{\infty }\left\Vert \pi
 	_{\lambda }\left( f\right) \right\Vert _{HS}^{2}\left\vert \lambda
 	\right\vert ^{n}d\lambda =\sum \int_{-\infty }^{\infty }\left\vert
 	\left\langle f,e_{\lambda}\left( h_{\alpha },h_{\beta }\right) \right\rangle
 	\right\vert ^{2}\left\vert \lambda \right\vert ^{n}d\lambda
 	\end{equation*}
 	where $\left\Vert \cdot \right\Vert _{HS}$ denotes the Hilbert-Schmidt norm, and the
 	sum runs on $h_{\alpha },h_{\beta }\in \mathcal{B}$. Notice that $\left\langle \pi _{\lambda }\left( t,v\right) \pi _{\lambda }\left( f\right) h,h'\right\rangle =\left( f\ast e_{\lambda }\left( h,h'\right)
 	\right) \left( t,v\right)$, then
 	\begin{eqnarray*}
 		f\left( t,v\right) &=&\sum_{j\in \Lambda }\int_{-\infty }^{\infty
 		}\sum_{h_{\alpha }\in \mathcal{B}_{j}}\left( f\ast e_{\lambda }\left(
 		h_{\alpha },h_{\alpha }\right) \right) \left( t,v\right) \left\vert \lambda
 		\right\vert ^{n}d\lambda\\
 		&=&\sum_{j\in \Lambda }\int_{-\infty }^{\infty }\left( f\ast \phi _{\lambda
 			,j}\right) \left( t,v\right) \left\vert \lambda \right\vert ^{n}d\lambda
 	\end{eqnarray*}
 	where the last equality follows from \eqref{traza}. By a straightforward computation we obtain that
 	\begin{equation} \label{formula1}
 	f\ast e_{\lambda }\left( h_{\alpha },h_{\alpha }\right) =\sum_{h_{\beta }\in 
 		\mathcal{B}}\left\langle f,e_{\lambda}\left( h_{\alpha },h_{\beta }\right)
 	\right\rangle _{L^{2}\left( H_{n}\right) }e_{\lambda}\left( h_{\alpha },h_{\beta
 	}\right).
 	\end{equation}
 	By \eqref{formula1} it follows that $f\ast e_{\lambda }\left( h_{\alpha },h_{\alpha }\right)
 	\in H_{\lambda ,j}$ and $||f\ast e_{\lambda }( h_{\alpha },h_{\alpha })|| _{\lambda ,j}^{2}=\sum_{h_{\beta }\in \mathcal{B}} |
 	\left\langle f,e_{\lambda}( h_{\alpha },h_{\beta }) \right\rangle | ^{2}$, for all $h_{\alpha
 	}\in \mathcal{B}_{j}$. Hence, we obtain that the orthogonal projection $Q_{\lambda ,j}\left( f\right) =f\ast \phi _{\lambda ,j}$ maps $L^{2}\left(
 	H_{n}\right) $ onto $H_{\lambda ,j},$ $H_{\lambda ,j}$ is irreducible, and
 	by \eqref{formula1} $\left\Vert f\right\Vert ^{2}=\sum_{j\in \Lambda }\int_{-\infty
 	}^{\infty }\left\Vert Q_{\lambda ,j}f\right\Vert _{\lambda ,j}^{2}\left\vert
 	\lambda \right\vert ^{n}d\lambda .$ This concludes the proof of the theorem.
 \end{proof}

\section{Description of $\phi _{\lambda,j}.$}

In this section we describe the set $\mathfrak{B}$ of spherical functions corresponding
to the set of generic (or with full Plancherel measure) representations of $%
N\left( \mathfrak{g},V\right) .$ These computations involve integration on $%
G/T,$ which is difficult to carry out with exception of a few cases. Nevertheless, we obtain a parametrization of $\mathfrak{B.}$

As we saw before the metaplectic representation $\varpi_{\lambda}$ of $K_{\lambda}$ is given by \eqref{metaplectica}. We assume that it decomposes into irreducible components as $\mathcal{P}(V) = \oplus_{j \in \Lambda} W_{\lambda,j}$. We also saw that 
\begin{equation*}
\rho _{\lambda }\left( z,v\right) =e^{i\left\vert \lambda \right\vert
	\left\langle z,Y_{\lambda }\right\rangle }\pi _{\left\vert \lambda
	\right\vert }\left( 0,v\right) .
\end{equation*}
The set of spherical functions of $(K_{\lambda},H_n)$ corresponding to the Fock representations $\pi_{|\lambda|}$ is given by $\{\psi_{\lambda, j} \}_{j \in \mathbb{N}\cup \{0\}}$, where $\psi_{\lambda,j}$ is described in \eqref{numerito}.
As
\begin{eqnarray*}
 \rho _{\lambda }\left( Ad\left( g\right) z,\pi \left( g\right)
v\right) &=&\pi _{\left\vert \lambda \right\vert }\left( \left\langle Ad\left(
g\right) z,Y_{\lambda }\right\rangle ,0\right) \pi _{\left\vert \lambda
	\right\vert }\left( 0,\pi \left( g\right) v\right)\\
&=&e^{i\left\vert \lambda
	\right\vert \left\langle z,Ad\left( g^{-1}\right) Y_{\lambda }\right\rangle
}\pi _{\left\vert \lambda \right\vert }\left( 0,\pi \left( g\right) v\right),
\end{eqnarray*}
we obtain that

\begin{equation*}
\int_{G/T} \psi_{\lambda,j}(g \cdot(z,v)) \,
d \dot{g}=\int_{G/T}e^{i\left\vert \lambda \right\vert \left\langle z,Ad\left(
	g^{-1}\right) Y_{\lambda }\right\rangle } \psi_{\lambda,j}(0,\pi(g)v) \, d \dot{g}.
\end{equation*}

 By the description
in \cite{BJR2} of the set of bounded spherical functions of a
Gelfand pair $\left( K,H_{n}\right) $ we know that for $(t,v) \in H_n$ and  $\lambda > 0$
 $$\psi _{
	\lambda ,j}\left( t,v\right) =e^{it \lambda}q_{j}\left( \lambda ^{\frac{1}{2}
}v\right) e^{-\frac{\lambda  }{4}\left\vert
	v\right\vert ^{2}},$$ where $q_{j}$ is a \textit{real} $K$-invariant polynomial. Indeed, assume $\lambda=1$ and
let $\mathcal{P}\left( V\right) ^{\mathbb{R}}$ denote the algebra of 
real $K$-invariant polynomials. Then it is  proved in \cite{BJR2}  that
there is a canonical basis $\left\{ p_{j}\right\} _{j\in \Lambda }$ of the
vector space $\mathcal{P}\left( V\right) ^{\mathbb{R}}$, $p_{j}\in W_{j}:=W_{j,1},$ such that the sequence $\left\{
q_{j}\right\} _{j\in \Lambda }$ is obtained from $\left\{ p_{j}\right\}
_{j\in \Lambda }$by applying the Gram Schmidt process with respect to the
measure $e^{-\frac{1}{4}\left\vert v\right\vert ^{2}}dv$. Thus 
\begin{equation*}
\phi _{\lambda ,j}\left( z,v\right) =e^{-\frac{\left\vert \lambda
		\right\vert }{4}\left\vert v\right\vert ^{2}}\left( \int_{G/T}e^{i\left\vert
	\lambda \right\vert \left\langle z,Ad\left( g^{-1}\right) Y_{\lambda
	}\right\rangle }q_{j}\left( \left\vert \lambda \right\vert ^{\frac{1}{2}}\pi
\left( g\right) v\right) d \dot{g}\right) .
\end{equation*}

\bigskip

\begin{obs}
In the case that $\mathfrak{g}= \mathfrak{c} \, \oplus \, \mathfrak{g}'$, and $Y_{\lambda}= Z_{\lambda} + Y'_{\lambda}, \ Z_{\lambda} \in \mathfrak{c}, \ Y_{\lambda}' \neq 0$, we have that $$\phi_{\lambda,j}(z,v)=e^{i |\lambda| \langle z, Z_{\lambda} \rangle } \phi_{\lambda',j} (z,v),$$
where $|\lambda'| = |\lambda| |Y_{\lambda}'|$.
\end{obs}

In the following, we analize the set $\mathcal{B}$ case by case. We denote by  $T_n$ the $n$-dimensional torus.

\medspace

$\bullet $ \textbf{Case I}. In this case $\mathfrak{g=su}\left( 2\right) ,V=\mathbb{H}%
^{n} $ and $\mathfrak{n}= \mathfrak{su}(2) \oplus \mathbb{H}^{n}$. \ $\mathfrak{su}\left( 2\right) $ is isomorphic to $\text{Im}(\mathbb{H})
$ and the action is given by $q$.$\left( v_{1},...,v_{n}\right) =\left(
qv_{1},...,qv_{n}\right) ,$ for $q\in \text{Im}(\mathbb{H})$, $v=\left(
v_{1},...,v_{n}\right) \in \mathbb{H}^{n}.$

Thus $\mathfrak{n=}$ $\text{Im}(\mathbb{H}) \oplus \mathbb{H}^{n}$ is a Lie algebra of
Heisenberg type$,$ $K=SU\left( 2\right) \times Sp\left( n\right) $ and $%
K_{\lambda }=T_{1}\times Sp\left( n\right) $ where 
\begin{equation}\label{T_1}
T_{1}:=\lbrace{ 
\left(
\begin{array}{lr}
e^{i\theta} & 0 \\
0 & e^{-i\theta} \\

\end{array}
\right)
: \,\theta\in \mathbb{R}
\rbrace}
\end{equation}
 is a maximal torus of $SU\left( 2\right)$. It is well known that the natural action on the space $\mathcal{%
	P}_{j}\left( \mathbb{C}^{2n}\right) $ of homogeneous polynomial of degree $j$ is irreducible and we denote it by $\eta_j$. Then the metaplectic representation of $K_{\lambda}$ acting on $\mathcal{P}\left( \mathbb{C}^{2n}\right) $ decomposes as 
$$\varpi \downarrow
K_{\lambda } = \oplus _{j=0}\chi _{j} \otimes \eta_j,$$
 where $\chi _{j}\left( \theta \right) =e^{-ij\theta }$.
 
 It is also well known that $\psi _{\lambda,j}(t,v) = e^{it\left\vert \lambda \right\vert }L_{j}^{2n-1}\left( 
\frac{\left\vert \lambda \right\vert }{2}\left\vert v\right\vert ^{2}\right)
e^{-\frac{\left\vert \lambda \right\vert }{4}\left\vert v\right\vert ^{2}}$  where $L_j^{2n-1}$ is a Laguerre polynomial of degree $j$, (see \cite{Fo} page 64). Thus, $$\phi _{\lambda ,j}\left( z,v\right) =\int_{SU\left( 2\right) /T_1}\psi
_{\lambda ,j}\left( \left\langle Ad\left(
g^{-1}\right) Y_{\lambda },z\right\rangle ,g.v\right) d\dot{g}.$$
Since $Ad:SU\left( 2\right) \rightarrow SO\left( 3\right) $ is a surjective
morphism with kernel $\pm 1$ and $SO\left( 3\right) /SO\left( 2\right) $ is
homeomorphic to the two dimensional sphere $S^{2},$ we have that
\begin{eqnarray*}
\phi _{\lambda ,j}\left( z,v\right) &=& \int_{SO\left( 3\right) /SO\left(
	2\right) }\psi _{\lambda ,j}\left( \left\langle g^{-1}Y_{\lambda
},z\right\rangle ,g.v\right) d\dot{g} \\
&=&\left( \int_{S^{2}}e^{i\left\vert \lambda \right\vert \xi .z}d\xi \right)
L_{j}^{2n-1}\left( \frac{\left\vert \lambda \right\vert }{2}\left\vert
v\right\vert ^{2}\right) e^{-\frac{\left\vert \lambda \right\vert }{4} \left\vert v\right\vert ^{2}}\\
&=&J_{\frac{1}{2}}\left( \left\vert \lambda
\right\vert z\right) L_{j}^{2n-1}\left( \frac{\left\vert \lambda \right\vert 
}{2}\left\vert v\right\vert ^{2}\right) e^{-\frac{\left\vert \lambda
		\right\vert }{4}\left\vert v\right\vert ^{2}},
\end{eqnarray*}
where $d\xi $ denotes the $SO\left( 3\right) $-invariant measure on $S^{2},$ and $J_{\frac{1}{2}}$ is
the Bessel function of order $\frac{1}{2}$ of the first kind.

\medskip

$\bullet $ \textbf{Case II}. In this case $N$ does not have square integrable representations.

\medskip

$\bullet $ \textbf{Case III}. In this case $\mathfrak{g=su}\left( 2\right) \oplus \mathfrak{su}\left( 2\right) ,V=\mathbb{H}^{k_{1}}\oplus \mathbb{R}^{4}\oplus \mathbb{H}^{k_{2}}$ and $\mathfrak{n}= \mathfrak{su}\left( 2\right) \oplus \mathfrak{su}\left( 2\right) \oplus \mathbb{H}^{k_{1}}\oplus \mathbb{R}^{4}\oplus \mathbb{H}^{k_{2}}$. The first copy (resp. the second) of $%
\mathfrak{su}\left( 2\right) $ acts as $\mathfrak{sp}\left( 1\right) $ on $\mathbb{H}^{k_{1}}$ \ and trivially on $\mathbb{H}^{k_{2}}$ (resp. on $\mathbb{H}^{k_{2}}$ and trivially on $\mathbb{H}^{k_{1}})$, and as $%
\mathfrak{so}\left( 4\right) $ on $\mathbb{R}^{4}.$ Thus $U=Sp\left(
k_{1}\right) \times Sp\left( k_{2}\right) ,K=Spin\left( 4\right) \times U$ and $K_{\lambda } = T_{2}\times U,$ where 
\begin{equation}\label{T_2}
T_{2}:=\left\{ 
	\left(
	\begin{array}{lr}
	e^{i\theta_1} & 0 \\
	0 & e^{i\theta_2} \\
	
	\end{array}
	\right)
	: \,\theta_1, \theta_2 \in \mathbb{R}
\right\}
\end{equation}
 is a maximal torus of $Spin\left( 4\right)$. Since $\mathcal{P}\left(
 V\right) =\mathcal{P}\left( \mathbb{C}^{2k_{1}}\right) \otimes \mathcal{P}%
 \left( \mathbb{C}^{2}\right) \otimes \mathcal{P}\left( \mathbb{C}%
 ^{2k_{2}}\right) $, we can decompose the metaplectic representation as
 $$\varpi \downarrow K_{\lambda }=\left( \oplus_{j \geq 0} \, \chi_{j}\left( \theta_{1}\right) \eta^{k_1}_{j} \right) \otimes \left( \oplus_{l_{1},l_{2} \geq 0} \, \chi_{l_{1},l_{2}}\left( \theta _{1},\theta _{2}\right) \right) \otimes \left( \oplus_{s \geq 0} \, \chi_{s}\left( \theta _{2}\right) \eta^{k_2}_{s} \right),$$
 where $\eta^{k_i}_j$ is the natural action on the space $\mathcal{P}_j(\mathbb{C}^{2k_i})$.\\
 
 We know that $\psi _{\lambda ,j,l_{1},l_{2},s}\left( t,v\right) =\sum_{\alpha}
 \left\langle \pi _{\left\vert \lambda \right\vert }\left( t,v\right)
 h_{\alpha },h_{\alpha }\right\rangle $ where $\left\{ h_{a}\right\} $ is a
 basis of the irreducible component $W_{j,l_{1},l_{2},s}$ . Writing $v=\left(
 v_{1},u,v_{2}\right) ,$ with $v_{1}\in \mathbb{C}^{2k_{1}}, v_{2}\in \mathbb{C%
 }^{2k_{2}}$ and $u=\left( u_{1},u_{2}\right) \in \mathbb{C}^{2},$ an easy computation shows that
 $$\pi _{\left\vert \lambda \right\vert }\left( t,v\right) =\pi _{\left\vert
 	\lambda \right\vert }\left( \frac{t}{3},\left( v_{1},0,0\right) \right)
 \otimes \pi _{\left\vert \lambda \right\vert }\left( \frac{t}{3},\left(
 0,u,0\right) \right) \otimes \pi _{\left\vert \lambda \right\vert }\left( 
 \frac{t}{3},\left( 0,0,v_{2}\right) \right), $$ and since the trace of a tensor product is the product of the traces, we obtain 
 $$ \psi _{\lambda ,j,l_{1},l_{2},s}\left( t,v\right) =e^{i\left\vert \lambda \right\vert t} \, L(v),$$
 \begin{equation*}
 L(v) = L_{j}^{2k_{1}-1}\left( \frac{\left\vert \lambda \right\vert }{2%
 }\left\vert v_{1}\right\vert ^{2}\right) L_{l_{1}}^{0}\left( \frac{%
 	\left\vert \lambda \right\vert }{2}\left\vert u_{1}\right\vert ^{2}\right)
 L_{l_{2}}^{0}\left( \frac{\left\vert \lambda \right\vert }{2}\left\vert
 u_{2}\right\vert ^{2}\right) L_{s}^{2k_{2}-1}\left( \frac{\left\vert \lambda
 	\right\vert }{2}\left\vert v_{2}\right\vert ^{2}\right) e^{-\frac{\left\vert
 		\lambda \right\vert }{4}\left\vert v\right\vert ^{2}},
 \end{equation*}
where $L^{n}_k$ is the Laguerre polynomial of degree $k$.\\
 
 On the first hand, $SU\left( 2\right) $ is acting as $Im(\mathbb{H})$ (or $Sp\left(1\right)$) on each component of $\mathbb{C}^{2k_{i}}$, $i=1,2$. On the other hand, $G \simeq Sp\left( 1\right) \times Sp\left( 1\right) $ acts on $\mathbb{C}%
 ^{2}$ $\simeq \mathbb{H}$ by the rule: $v \longmapsto g.v=q_{1}vq_{2}$ $,$
 for $g=\left( q_{1},q_{2}\right) \in G,v\in \mathbb{H}$ . Then we have that 
 $$L_{j}^{2k_{1}-1}\left( \frac{\left\vert \lambda \right\vert }{2}\left\vert
 gv_{1}\right\vert ^{2}\right) =L_{j}^{2k_{1}-1}\left( \frac{\left\vert
 	\lambda \right\vert }{2}\left\vert v_{1}\right\vert ^{2}\right) ,\bigskip
 L_{s}^{2k_{2}-1}\left( \frac{\left\vert \lambda \right\vert }{2}\left\vert
 gv_{2}\right\vert ^{2}\right) =L_{s}^{2k_{2}-1}\left( \frac{\left\vert
 	\lambda \right\vert }{2}\left\vert v_{2}\right\vert ^{2}\right) ,$$ and the spherical functions are given by
 \begin{equation*}
 \phi _{\lambda ,j,l_{1},l_{2},s}\left( z,v\right) =\left(
 \int_{G/T}e^{i\left\vert \lambda \right\vert \left\langle Ad\left(
 	g^{-1}\right) Y_{\lambda },z\right\rangle }L_{l_{1}}^{0}\left( \frac{%
 	\left\vert \lambda \right\vert }{2}\left\vert \left( gu\right)
 _{1}\right\vert ^{2}\right) L_{l_{2}}^{0}\left( \frac{\left\vert \lambda
 	\right\vert }{2}\left\vert \left( gu\right) _{2}\right\vert ^{2}\right)
 dg\right) \times
 \end{equation*}
 
 \begin{equation*}
 e^{-\frac{\left\vert \lambda \right\vert }{4}\left\vert v\right\vert 
 	^{\substack{  \\ 2}}}L_{j}^{2k_{1}-1}\left( \frac{\left\vert \lambda
 	\right\vert }{2}\left\vert v_{1}\right\vert ^{2}\right)
 L_{s}^{2k_{2}-1}\left( \frac{\left\vert \lambda \right\vert }{2}\left\vert
 v_{2}\right\vert ^{2}\right)
 \end{equation*}
 where $g\left( u_{1},u_{2}\right) =\left( \left( gu\right) _{1},\left(
 gu\right) _{2}\right) $.
 
 \medskip

$\bullet $ \textbf{Case IV}. Here $\mathfrak{g=sp}\left( 2\right) ,V$ $=$ $%
\left( \mathbb{H}^{2}\right) ^{n}$ and $\mathfrak{n=sp}\left( 2\right)
\oplus \left( \mathbb{H}^{2}\right) ^{n}$. The real action is given by $\pi
\left( g\right) \left( v_{1},...,v_{n}\right) =$ $\left(
gv_{1},...,gv_{n}\right) ,v_{j}\in \mathbb{H}^{2}$ for $j=1,...,n$. By the Schur's Lemma, the group of orthogonal intertwining operator is isomorphic to $Sp(n)$ with the action on $(\mathbb{H}^{2})^{n}$ given by the $2n \times 2n$ matrix $a_{ij}I$, $a_{ij} \in \mathbb{H}$, and $I$ is the $2 \times 2$ identity. 
Thus $K=Sp\left( 2\right) \times Sp\left( n\right) ,$ and $K_{\lambda
}=T_{2}\times Sp\left( n\right) $ where $T_2$ is a maximal torus of $Sp\left( 2\right)$ as in \eqref{T_2}.\\

Writing $\left( v_{1},...,v_{n}\right) =\left( \left( u_{1},w_{1}\right)
,...,\left( u_{n},w_{n}\right) \right) ,$ with $\left( u_{j},w_{j}\right)
\in \mathbb{H}^{2}$ for $j=1,...n,$ we have that the action of $Sp\left(n\right) $ is given by $$g\left( \left( u_{1},w_{1}\right) ,...,\left( u_{n},w_{n}\right) \right)
=\left( g\left( u_{1},...,u_{n}\right) ,g\left( w_{1},...,w_{n}\right)
\right).$$ So the action of $Sp\left( n\right) $ on $\mathcal{P}\left( 
\mathbb{C}^{4n}\right) $ splits as $\mathcal{P}\left( \mathbb{C}^{2n}\right)
\otimes \mathcal{P}\left( \mathbb{C}^{2n}\right) =\oplus _{r,s}\mathcal{P}%
_{r}\left( \mathbb{C}^{2n}\right) \otimes \mathcal{P}_{s}\left( \mathbb{C}%
^{2n}\right) $. On the other hand, $T_{2}$ acts naturally on each $\mathbb{H}^{2}$ by $$
\left(
\begin{array}{cc}
e^{i\theta _{1}} & 0 \\ 
0 & e^{i\theta _{2}}
\end{array}
\right) (u_j,v_j) = (e^{-i \theta_{1}} u_j, e^{-i \theta_{2}} v_j).
$$

 As above we denote by $\eta _{j}$ the irreducible representation of $Sp\left( n\right) $
on $\mathcal{P}_{j}\left( \mathbb{C}^{2n}\right) $. In \cite{KT} it is proved that $\eta_s \otimes \eta_r = \oplus_{j=0}^{s} \oplus_{i=0}^{j} \eta_{(r+s-j-i,j-i)}$ where $\eta_{(r+s-j-i,j-i)}$ is the irreducible representation of $Sp(n)$ with highest weight $(r+s-j-i,j-i,0, \cdots ,0)$, for $r \geq s$.  Then
\begin{equation*}
\varpi \downarrow K_{\lambda }=\oplus _{r,s}\chi _{r,s}\left( \theta
_{1},\theta _{2}\right) \otimes (\oplus _{j=0}^{s}\oplus _{i=0}^{j}\eta
_{\left( r+s-j-i,j-i\right) }),
\end{equation*}%
where $\chi _{r,s}\left( \theta _{1},\theta _{2}\right) =e^{-i\left( r\theta
	_{1}+s\theta _{2}\right) }$.\\

The polynomial $q_{r,s,j,i}$ in $\mathcal{P}\left( \mathbb{C}^{4n}\right) $
corresponding to the spherical function $\psi _{\lambda,r,s,j,i}$ is $K_{\lambda }-$%
invariant, thus $q_{r,s,j,i}\left( t,v\right) =q_{r,s,j,i}\left(
t,\left\vert u\right\vert ^{2},\left\vert w\right\vert ^{2}\right) ,$ but $%
Sp\left( 2\right) $ preserves the norm of $\left( u_{j},w_{j}\right) $ for $%
j=1,...n.$ Thus 
 $$\phi _{\lambda
	,r,s,j,i}\left( z,v\right) = e^{- \frac{|\lambda|}{4} |v|^{2}} \left(\int_{G/T} e^{i|\lambda| \langle Ad(g^{-1})Y_{\lambda},z \rangle} q_{r,s,j,i} (t,g \cdot v) \, d\dot{g} \right).$$\\

$\bullet $ \textbf{Case V}. In this case $\mathfrak{g=su}\left(
n\right) ,V=\mathbb{C}^{n}$, $\mathfrak{n=su}\left( n\right) \oplus \mathbb{C%
}^{n}$ and $\pi $ is the canonical action of $\mathfrak{su}\left( n\right) $ on $\mathbb{C}^{n}.$ Since it is irreducible, the group of the orthogonal
intertwining operators is a one dimensional torus which we denote by $T_{1}.$
So $K=SU\left( n\right) \times T_{1},$ $K_{\lambda }=T_{n-1}\times T_{1}$ where $T_{n-1}$ is a maximal torus of $SU\left( n\right) $, and
$$\varpi \downarrow K_{\lambda }=\oplus_{m_{1},...,m_{n} \in \mathbb{Z\geq }%
	0} \ \chi _{m_{1},...,m_{n}},$$ where $\chi _{m_{1},...,m_{n}} \ \left( \theta
_{1},...,\theta _{n}\right) =e^{-i\left( m_{1}\theta _{1}+...+m_{n}\theta
	_{n}\right) }$.

The set of spherical functions corresponding to the pair $\left(
T_{n},H_{n}\right) $ were computed in \cite{Fo}:
\begin{equation}\label{caseV}
\psi _{\lambda ,m_{1},...,m_{n}}\left( t,v\right)
=e^{i\left\vert \lambda \right\vert t} \prod_{j=1}^{n} L_{m_{j}}^{0}\left( 
\frac{\left\vert \lambda \right\vert }{2}\left\vert v_{j}\right\vert
^{2}\right) e^{-\frac{\left\vert \lambda \right\vert }{4}\left\vert
	v\right\vert ^{2}},
\end{equation}
 and setting  $gv=\left(
\left( gv\right) _{1},...,\left( gv\right) _{n}\right)$ for $g\in SU\left( n\right)$, we obtain the following expression for the set of generic spherical functions
$$\phi _{\lambda ,m_{1},...,m_{n}}\left( z,v\right)
=e^{-\frac{\left\vert \lambda \right\vert }{4}\left\vert v\right\vert
	^{2}}\left( \int_{SU\left( n\right) /T_{n}}e^{i\left\vert \lambda
	\right\vert \left\langle Ad\left( g^{-1}\right) Y_{\lambda },z\right\rangle
}\prod_{j=1}^{n}L_{m_{j}}^{0}\left( \frac{\left\vert \lambda \right\vert }{%
	2}\left\vert \left( gv\right) _{j}\right\vert ^{2}\right) d\dot{g} \right), $$\\
with $L^{0}_{m_j}$ is the Laguerre polynomial of degree $m_j$.\\

$\bullet $ \textbf{Case VI}. In this case $\mathfrak{g=so}\left( 2n\right)
,V=\mathbb{R}^{2n},$ $\mathfrak{n=so}\left( 2n\right) \oplus \mathbb{R}^{2n}$
and $\pi $ is the canonical action of $\mathfrak{so}\left( 2n\right) $ on $\mathbb{R}%
^{2n}.$ Since it is irreducible, the group of the orthogonal intertwining
operators is trivial. Thus $K=SO\left( 2n\right) $ and $K_{\lambda }$ is an $%
n$ dimensional torus.


 As in the above case, the metaplectic representation is decomposed into a direct sum of characters without multiplicity  as
 $$\varpi \downarrow_{T_n} = \bigoplus_{(m_1, \cdots, m_n) \in \mathbb{Z} \geq 0} \ \chi_{(m_1, \cdots, m_n)},$$
where $(\chi_{(m_1, \cdots, m_n)}(\theta_1, \cdots , \theta_n))(z_1^{m_1} \cdots z_n^{m_n}) = e^{-im_1 \theta_1} z_1^{m_1} \cdots e^{-im_n \theta_n} z_n^{m_n}.$
 
Then $\psi _{\lambda
	,m_{1},...,m_{n}}$ is given by \eqref{caseV}, but here we
have to integrate on $SO\left( 2n\right) /T_{n}$, that is,
$$\phi _{\lambda ,m_{1},...,m_{n}}\left( z,v\right)
=e^{-\frac{\left\vert \lambda \right\vert }{4}\left\vert v\right\vert
	^{2}}\left( \int_{SO\left( 2n\right) /T_{n}}e^{i\left\vert \lambda
	\right\vert \left\langle Ad\left( g^{-1}\right) Y_{\lambda },z\right\rangle
}\prod_{j=1}^{n}L_{m_{j}}^{0}\left( \frac{\left\vert \lambda \right\vert }{%
	2}\left\vert \left( gv\right) _{j}\right\vert ^{2}\right) d\dot{g} \right),$$\\
where $L^{0}_{m_j}$ is the Laguerre polynomial of degree $m_j$.\

\medskip

$\bullet $ \textbf{Case VII}. In this case $N(\mathfrak{g},V)$ is the $(2n+1)$-dimensional Heisenberg group, $K=U(n)$ and the set of bounded spherical functions was described by many authors, see for example \cite{BJR2} and \cite{Fo}.

\medskip

$\bullet $ \textbf{Case VIII}. In this case $\mathfrak{g=u}\left( 2\right) ,V=\left( 
\mathbb{C}^{2}\right) ^{k}\oplus \left( \mathbb{C}^{2}\right) ^{n}$ and  $\mathfrak{n} = \mathfrak{u}(2) \oplus \left( 
\mathbb{C}^{2}\right) ^{k}\oplus \left( \mathbb{C}^{2}\right) ^{n}$. $\mathfrak{u}\left( 2\right) $ acts in the following way:

- on each of the $k$ components of $\left( \mathbb{C}^{2}\right) ^{k}$ it acts in the
natural way, and

- in $\left( \mathbb{C}^{2}\right) ^{n}$ the center of $\mathfrak{u}\left( 2\right) $
acts trivially and the semisimple part acts as $\mathfrak{sp}\left( 1\right) 
$ (or $\text{Im}(\mathbb{H})$) on the left side on each of the $n$ components
of $\left( \mathbb{C}^{2}\right) ^{n}.$\\

For  $n$ positive, $K=SU\left( 2\right) \times U\left( k\right) \times Sp\left( n\right)$, $K_{\lambda }=T_{1}\times U\left( k\right) \times Sp\left( n\right) $ is a maximal torus contained in $SU\left( 2\right)$. The action of $T_1$ on  $\mathcal{P}\left( (\mathbb{C}^{2})^{k}\right) $ is given  by $$p\left( u_{1},w_{1},...,u_{k},w_{k}\right) \rightarrow p\left( e^{i\theta
}u_{1},e^{-i\theta }w_{1},...,e^{i\theta }u_{k},e^{-i\theta }w_{k}\right)$$ for $ p \in \mathcal{P}(\mathbb{C}^{2k}), e^{i\theta} \in T_1, \, (u_i,w_i) \in \mathbb{C}^{2}.$
 Also $U\left( k\right) $ acts by a multiple of
the $2\times 2$ identity on each of the $k$ components of $\left( \mathbb{C}%
^{2}\right) ^{k}.$

We denote by $\nu_r$ (resp. $\eta_r$) the irreducible action of $U(k)$ (resp. $Sp(k)$) on $\mathcal{P}_r(\mathbb{C}^{2k})$.

As $U(k)$-module $\mathcal{P}(\mathbb{C}^{2k}) = \mathcal{P}(\mathbb{C}^{k}) \otimes \mathcal{P}(\mathbb{C}^{k}) = \oplus_{r,s} \nu_r \otimes \nu_s$. Moreover, $$\nu_r \otimes \nu_s = \oplus_{j=1}^{min(r,s)} \nu_{(r+s-2j,j)},$$ where $\nu_{(r+s-2j,j)}$ denotes the irreducible representation of highest weight $(r+s-2j,j,0,\cdots,0)$ (\cite{FH}, page 225). Thus $\mathcal{P}\left( V\right) =\mathcal{P}\left( \mathbb{C}^{2k}\right)
\otimes \mathcal{P}\left( \mathbb{C}^{2n}\right) $ and the decomposition of the metaplectic representation in irreducible components is
\begin{equation*}
\varpi \downarrow_{K_{\lambda}}= \left( \oplus _{r,s,j \in \mathbb{Z} \geq 0} \ \chi _{r-s}\left( \theta \right) \oplus_{j=1}^{min(r,s)} \nu_{(r+s-2j,j)} \right) \otimes \left( \oplus_{\mathit{l} \in \mathbb{Z} \geq 0} \ \chi_{\mathit{l}}(\theta) \, \eta_{\mathit{l}} \right). 
\end{equation*}

We set $v=\left( \mathbf{v}^{1},\mathbf{v}^{2} \right) ,\mathbf{v}^{1} \in \mathbb{C}^{2k}, \mathbf{v}^{2} \in 
\mathbb{C}^{2n}.$ As in case \textbf{III} by applying an elementary property of the trace of a linear map defined on a tensor product, we obtain the expression for the spherical function 
$$\psi _{ \lambda ,r,s,j,l}\left( t,v\right)
=e^{i\left\vert \lambda \right\vert t}e^{-\frac{\left\vert \lambda
		\right\vert }{4}\left\vert v\right\vert ^{\substack{  \\ 2}}}q_{j,r,s}\left( 
\frac{\left\vert \lambda \right\vert }{2} \mathbf{v}^{1}\right) L_{l}^{2n-1}\left( 
\frac{\left\vert \lambda \right\vert }{2}\left\vert \mathbf{v}^{2} \right\vert
^{2}\right) ,j=1,...,\min(r,s), \,l \geq 0,$$
where $L^{2n-1}_l$ is the Laguerre polynomial of degree $l$.

Let  $\mathbf{v}^{1}=\left( u_1,w_1,\cdots,u_k,w_k \right)\in \mathbb{C}^{2k} $, we have that $q_{j,r,s}$ is a polynomial in $|u|^{2}, |w|^{2}$ since $q_{j,r,s}$ is $U(k)$-invariant, but the action of $G=SU(2)$ is componentwise on each $(u_j,w_j)$. Thus
$$\phi_{\lambda,r,s,j,l}(t,v)= \left( \int_{G/T} e^{i|\lambda| \langle Ad(g^{-1})Y_{\lambda},z \rangle} q_{j,r,s}(\frac{|\lambda|}{2}g \cdot \mathbf{v}^{1}) \, d\dot{g} \right) e^{-(\frac{|\lambda|}{4} |v|^{2})} L^{2n-1}_l(\frac{|\lambda|}{2}|\mathbf{v}^{2}|^{2}).$$

As we observe in the proof of Theorem \eqref{teonil}, in this case there is no generic spherical functions associated to the center of $\mathfrak{g}$.\\

Case $n=0.$ To the set of spherical functions described above (with the
obvious changes since $n=0$) we add the set of spherical functions
corresponding to the elements of the center of $\mathfrak{g}.$ For this purpose, we need to decompose the metaplectic action of $K=$ $SU\left( 2\right)
\times U\left( k\right) $ on $\mathcal{P}\left( \mathbb{C}^{2k}\right) $.
We assume $k\geq 2$.

It is easy to see that $\mathbb{C}^{2k}$ is equivalent to $\mathbb{C}^{2}\otimes \mathbb{C}%
^{k} $ with the standard action  as $(SU(2)\times U(k))$-module. So we can apply the Corollary 5.2.8 from \cite{GW} to obtain the desired decomposition. Indeed, following the notation there,
let $\mathcal{F}_{2}^{\mu }$ be the irreducible representation of $Gl\left(
2,\mathbb{C}\right) $ with highest weight $\mu$ satisfying $\mu =\mu
_{1}\varepsilon _{1}+\mu _{2}\varepsilon _{2}$ where $\varepsilon_1, \varepsilon_{2}$ are the coordinate weights
 and $\mu _{1}+\mu _{2}=d$ with $\mu _{1} \geq \mu _{2}\geq
 0$ and let $\mathcal{F}_{k}^{\mu }$ be the irreducible
representation of $Gl\left( k,\mathbb{C}\right) $ with highest weight the $%
k$-tuple $\left( \mu _{1},\mu _{2},0,...,0\right)$. Then we have that the
homogeneous polynomials of degree $d$ over $\mathbb{C}^{2}\otimes \mathbb{C}^{k}$ decomposes as 
\begin{equation*}
\mathcal{P}_{d}\left( \mathbb{C}^{2}\otimes \mathbb{C}^{k}\right) =\bigoplus
_{\mu }\mathcal{F}_{2}^{\mu }\otimes \mathcal{F}_{k}^{\mu }
\end{equation*}
such that $\mu _{1}+\mu _{2}=d$. Now in order to restrict to $Sl\left( 2,\mathbb{C}
\right) \times Gl\left( k,\mathbb{C}\right)$, let $l=\mu _{1}-\mu _{2}$ and
let $\mathcal{F}^{l}$ be the irreducible representation of $Sl\left( 2,%
\mathbb{C}\right) $ of dimension $l+1$. Then, the restriction to $Sl\left( 2,%
\mathbb{C}\right) \times Gl\left( k,\mathbb{C}\right) $ decomposes as 
\begin{equation*}
\mathcal{P}\left( \mathbb{C}^{2}\otimes \mathbb{C}^{k}\right) =\oplus _{l,d}%
\mathcal{F}^{l}\otimes \mathcal{F}_{k}^{\mu },l,d\geq 0
\end{equation*}%
with $\mu _{1}+\mu _{2}=d$ and $\mu _{1}-\mu _{2}=l$.
The corresponding set of spherical functions is given by \eqref{numerito}
and it is parametrized by $\left\{ \lambda ,d,l\right\} $ with $\lambda \neq
0$ and $l, d\geq 0.$

\medskip

 $\bullet $ \textbf{Case IX}. In this case $\mathfrak{g}=\mathfrak{u}(n)$, $V=\mathbb{C}^{n}$ and the action is the standard.  Thus $G=SU(n)$, $U=T^{1}$ and $K=U(n)$. For $X_{\lambda} \in \mathfrak{g}'$ the corresponding spherical functions are as in the \textbf{case V}.\\
 
 For the elements in the center of $\mathfrak{g}$, the stabilizer is $K$ and 
 $$\omega \downarrow _K = \oplus_{r \geq 0} v_r$$
 where $v_r$ denotes the irreducible representation of $U(n)$ on the homogeneous polynomials of degree $r$. The set of associated spherical functions are $$\phi_{\lambda,r}(t,v) = e^{i|\lambda|t} L_r^{n-1} \left( \frac{\lambda}{2} |v|^{2} \right) e^{-\frac{|\lambda|}{4}|v|^{2}},$$
 where $L^{n-1}_r$ is the Laguerre polynomial of degree $r$.
 
 \medskip

$\bullet $ \textbf{Case X}. In this case $\mathfrak{g=su}\left( m_{1}\right)
\, \oplus ...\oplus \, \mathfrak{su}\left( m_{r}\right)  \oplus  \mathfrak{c}$ where  
$\mathfrak{c}$ is its center and there are $\alpha$ copies of $\mathfrak{su}(2)$. The abelian component satisfies $ 1 \leq dim (\mathfrak{c})\leq r-1$; $V=V_{1}\oplus
...\oplus V_{r},$ and the representation $\pi $ of $\mathfrak{g}$ on $V$ is
defined in the following way:

For each $1\leq j\leq r,\mathfrak{su}\left( m_{j}\right) $ acts non
trivially only on $V_{j}$, $\mathfrak{c}$ has a unique subspace $\mathfrak{c}%
_{j}$ acting non trivially on $V_{j}$ and $\dim (\mathfrak{c}_{j})=1$. If $%
m_{j}\geq 3, \, V_{j}=\mathbb{C}^{m_{j}}$ and $\mathfrak{su}\left( m_{j}\right)
\oplus \mathfrak{c}_{j}$ (which is isomorphic to $\mathfrak{u}\left( m_{j}\right) $)
acts in the standard way on $V_{j};$ we denote by $S^{1}$ the group of
intertwining operators of this action. If $m_{j}=2,V_{j}=\left( \mathbb{C}%
^{2}\right) ^{k}\oplus \left( \mathbb{C}^{2}\right) ^{n}$ and

$\mathfrak{su}\left( 2\right) \oplus \mathfrak{c}_{j}$ acts on $V_{j}$ as in
case \textbf{VIII}, therefore the group of intertwining operators is $U\left( k\right) \times Sp\left( n\right) .$

We first consider the case $\mathfrak{g=su}\left( m\right) \oplus \mathfrak{su}\left( 2\right) \oplus \mathfrak{c},m\geq 3, \, n > 0$. Thus $\dim \mathfrak{c=1,}$ $%
V=\mathbb{C}^{m}\oplus \left( \mathbb{C}^{2}\right) ^{k}\oplus \left( 
\mathbb{C}^{2}\right) ^{n}$, $G=SU\left( m\right) \times SU\left( 2\right) $, $
K=G\times S^{1}\times U\left( k\right) \times Sp\left( n\right) .$ Let $T_{m-1}$ be a maximal torus of $SU\left( m\right) ,$ thus $T_{m-1}\times S^{1}$
is ( isomorphic to) an $n$-dimensional torus acting on $\mathbb{C}^{m}$ in
the standard way, $K_{\lambda }=T_{m-1}\times S^{1}\times T_{1}\times
U\left( k\right) \times Sp\left( n\right) $ and $T_{1}\times U\left(
k\right) \times Sp\left( n\right) $ acts on $\left( \mathbb{C}^{2}\right)
^{k}\oplus \left( \mathbb{C}^{2}\right) ^{n}$ as in case \textbf{VIII}. Thus
\begin{equation*}
\varpi \downarrow K_{\lambda } =
\end{equation*}
\begin{equation*}
\left( \oplus _{k_{1},...,k_{m} \in \mathbb{Z\geq }0}\chi
_{k_{1},...,k_{m}}\left( \theta _{1},...,\theta _{m}\right) \right) \otimes
\oplus _{r,s}\oplus _{i=1}^{\min \left( r,s\right) }\chi _{r-s}\left( \theta
\right) \oplus _{i=1}^{\min (r,s)}\upsilon _{r+s-2i,i}\otimes \left( \oplus
_{j}\chi _{j}\left( \theta \right) \eta _{j}\right) .
\end{equation*}

For $v=\left( u,v_{1},v_{2}\right) ,u\in \mathbb{C}^{m},v_{1}\in 
\mathbb{C}^{2k},v_{2}\in \mathbb{C}^{2n}$, the spherical functions associated to the
pair $\left( K_{\lambda },N_{\lambda }\right)$ are
\begin{equation*}
\psi _{\lambda ,\mathbf{k,},r,s,i,j}\left( t,u,v_{1},v_{2}\right) = e^{i\left\vert \lambda \right\vert t}e^{-\frac{\left\vert \lambda
		\right\vert }{4}\left\vert v\right\vert ^{\substack{  \\ 2}}}\prod_{j=1}^{m}L_{k_{j}}^{0}\left( \frac{\left\vert \lambda \right\vert }{2}
\left\vert u_{j}\right\vert ^{2}\right) q_{i,r,s}\left( \frac{\left\vert
	\lambda \right\vert }{2}v_{1}\right) L_{j}^{2n-1}\left( \frac{\left\vert
	\lambda \right\vert }{2}\left\vert v_{2}\right\vert ^{2}\right),
\end{equation*}
where $\mathbf{k=}\left( k_{1},...,k_{m}\right) ,i=1,...,\min \left(
r,s\right) $ and $r,s,$ $j\geq 0.$

The action of $G$ is componentwise so writing $$q_{\mathbf{k},
	,r,s,i}\left( u,v_{1}\right) =\prod_{j=1}^{m} L_{k_{j}}^{0}\left( \frac{%
	\left\vert \lambda \right\vert }{2}\left\vert u_{j}\right\vert ^{2}\right)
q_{i,r,s}\left( \frac{\left\vert \lambda \right\vert }{2}v_{1}\right)$$
we have that the expression for the spherical function is
$$\phi _{\left\vert \lambda \right\vert ,\mathbf{k,},r,s,i,j}\left(
z,v\right) =\left( \int_{G/T}e^{i\left\vert \lambda \right\vert
	t\left\langle Ad\left( g\right) Y_{\lambda },z\right\rangle }q_{\mathbf{k,},r,s,i,}\left( \pi \left( g\right) \left( u,v_{1}\right) \right)
dg\right) L_{j}^{2n-1}\left( \frac{\left\vert \lambda \right\vert }{2}%
\left\vert v_{2}\right\vert ^{2}\right) e^{-\frac{\left\vert \lambda
		\right\vert }{4}\left\vert v\right\vert ^{\substack{  \\ 2}}}.$$
	
	\bigskip 

When $n=0$, the completion of the set of spherical functions corresponding to the elements of the center $\mathfrak{c}$ of $\mathfrak{g}$ follows from the case \textbf{IX} and case \textbf{VIII} with $n=0,$ since $U\left(
m\right) $ acts on $\mathcal{P}\left( \mathbb{C}^{m}\right) $ and $SU\left(
2\right) \times U\left( k\right) $ acts on $\mathcal{P}\left( \mathbb{C}%
^{2k}\right) .$ We have finished the simplest case.

\medskip

The general case follows similar lines.

\end{document}